\documentclass[11pt, reqno]{amsart}
\usepackage{amssymb, amsthm, amsmath, amsfonts,mathtools}
\usepackage{array, epsfig}
\usepackage{bbm}
\usepackage{enumitem}

\usepackage{hyperref}
\usepackage[numbers,square]{natbib}
\usepackage{color}
\usepackage{ stmaryrd }
\allowdisplaybreaks

\numberwithin{equation}{section}

  \evensidemargin
\oddsidemargin
\newcommand{\stirling}[2]{\genfrac{[}{]}{0pt}{}{#1}{#2}}
\newcommand{\stirlingsec}[2]{\genfrac{\{}{\}}{0pt}{}{#1}{#2}}

\newcommand{\stirlingb}[2]{B\hspace*{-0.2mm}\big[{#1},{#2}\big]}
\newcommand{\stirlingsecb}[2]{B\hspace*{-0.2mm}\big\{{#1},{#2}\big\}}

\newcommand{\N}{\mathbb{N}}

\newcommand{\R}{\mathbb{R}}

\newcommand{\1}{\mathbbm{1}}

\newcommand{\eps}{\varepsilon}

\newcommand*\xbar[1]{%
   \hbox{%
     \vbox{%
       \hrule height 0.5pt 
       \kern0.25ex
       \hbox{%
         \kern-0.05em
         \ensuremath{#1}%
         \kern-0.1em
       }%
     }%
   }%
}

\DeclareMathOperator{\E}{\mathbb{E}}

\DeclareMathOperator{\lin}{lin}
\DeclareMathOperator{\conv}{conv}
\DeclareMathOperator{\pos}{pos}

\DeclareMathOperator{\relint}{relint}

\newcommand{\bE}{\mathbb{E}}

\newcommand{\bP}{\mathbb{P}}

\newcommand{\cA}{\mathcal{A}}

\newcommand{\cC}{\mathcal{C}}

\newcommand{\cF}{\mathcal{F}}

\def\dint{\textup{d}}

\renewcommand{\P}{\mathbb{P}}

\newcommand{\aff}{\mathop{\mathrm{aff}}\nolimits}

\newcommand{\eqdistr}{\stackrel{d}{=}}

\newcommand{\bsl}{\backslash}
\newcommand{\ind}{\mathbbm{1}}

\theoremstyle{plain}
\newtheorem{theorem}{Theorem}[section]
\newtheorem{lemma}[theorem]{Lemma}
\newtheorem{corollary}[theorem]{Corollary}

\theoremstyle{definition}

\theoremstyle{remark}
\newtheorem{remark}[theorem]{Remark}

\makeindex
\makeglossary

\begin{document}

\author{Thomas Godland}
\address{Thomas Godland: Institut f\"ur Mathematische Stochastik,
Westf\"alische Wilhelms-Universit\"at M\"unster,
Orl\'eans-Ring 10,
48149 M\"unster, Germany}
\email{thomas.godland@uni-muenster.de}

\author{Zakhar Kabluchko}
\address{Zakhar Kabluchko: Institut f\"ur Mathematische Stochastik,
Westf\"alische {Wilhelms-Uni\-ver\-sit\"at} M\"unster,
Orl\'eans--Ring 10,
48149 M\"unster, Germany}
\email{zakhar.kabluchko@uni-muenster.de}

\title[Positive hulls of random walks and bridges]{Positive hulls of random walks and bridges}



\thanks{TG and ZK acknowledge support by the German Research Foundation under Germany's Excellence Strategy  EXC 2044 -- 390685587, Mathematics M\"unster: Dynamics - Geometry - Structure  and by the DFG priority program SPP 2265 Random Geometric Systems.}

\begin{abstract}
We study random convex cones defined as positive hulls of $d$-dimensional random walks and bridges.
We compute expectations of various geometric functionals of these cones such as the number of $k$-dimensional faces and the sums of conic quermassintegrals  of their $k$-dimensional faces.
These expectations are expressed in terms of Stirling numbers of both kinds and their $B$-analogues.
\end{abstract}

\maketitle
\tableofcontents

\section{Statement of the problem}
\label{sec:intro}

\subsection{Introduction}
A \textit{polyhedral cone} is an intersection of finitely many closed half-spaces in $\R^d$ whose bounding hyperplanes pass through the origin.  In other words,  polyhedral cones, (or  just \textit{cones}, for the purpose of the present paper) are sets of solutions to finite systems of linear homogeneous inequalities.  As such, they are fundamental objects appearing  in many areas of applied mathematics.  Equivalently, one can define a cone to be a \textit{positive hull} of finitely many vectors $x_1,\ldots,x_n\in \R^d$, denoted by
\begin{align*}
\pos\{x_1,\ldots,x_n\}=\{\lambda_1x_1+\ldots+\lambda_nx_n:\lambda_1,\ldots,\lambda_n\ge 0\}.
\end{align*}

We will be interested in \textit{random} cones.
The probably simplest way to construct such a  cone is to select $n$ independent random points $X_1,\ldots,X_n$ according to some probability density on $\R^d$ and to define
$$
P_n := \pos \{X_1,\ldots,X_n\}.
$$
A remarkable result of Wendel~\cite{jW62}, see also~\cite[Theorem~8.2.1]{SW08}, states that if the density of the $X_i$'s is symmetric with respect to the origin, then
$$
\P[P_n \neq \R^d] = \frac{1}{2^{n-1}} \sum_{k=0}^{d-1} \binom{n-1}{k}.
$$
For example, in the special case  with $n=4$ points in dimension $d=3$, this formula implies the value $7/8$ for the probability that a random tetrahedron with four vertices $X_1,\ldots,X_4$ does not contain the origin.  An elementary and elegant discussion of this special case can be found in~\cite{youtube}. Random cones of the form $P_n$, along with several other random cones related to $P_n$ by conditioning and taking the dual, were further studied by Cover and Efron~\cite{CoverEfron}, Donoho and Tanner~\cite{Donoho2009}, and Hug and Schneider~\cite{HugSchneider2016}. These authors computed explicitly expected values of some basic quantities associated with $P_n$ including the number of $k$-dimensional faces $f_k(P_n)$, the solid angle $\alpha(P_n)$, the conic intrinsic volumes $\upsilon_k(P_n)$, the conic quermassintegrals $U_k(P_n)$, and more general geometric functionals $Y_{m,l}(P_n)$ whose definition will be recalled below. More recently, asymptotic properties and threshold phenomena for these random cones as $n,d\to\infty$ have been investigated in~\cite{HugSchneiderThresholdPhenomena,HugSchneiderThresholdPhenomenaPart2,GKT2020_HighDimension1}.

In this paper we will be interested in random cones of the form $\pos\{S_1,\ldots,S_n\}$, where $S_k:= X_1+\ldots+X_k$ with $1\leq k\leq n$  is a random walk or a random bridge in $\R^d$ with increments $X_1,\ldots,X_n$.
Under appropriate assumptions on the increments, the probability that $\pos\{S_1,\ldots,S_n\} \neq  \R^d$ has been computed explicitly in the paper~\cite{KVZ15}; see~\eqref{eq:non-abs_prob_bridge} and~\eqref{eq:non-abs_prob_walk} below. In the one-dimensional case, this recovers a classical formula of Sparre Andersen~\cite[Theorem~1]{sparre_andersen_number_positive}  for the probability that a one-dimensional random walk (or bridge) does not change the sign. Furthermore, in~\cite{KVZ17}  an explicit formula for the expected $f$-vector of the \textit{convex} hulls of the form $\conv\{S_1,\ldots,S_n\}$ has been derived. Earlier, this type of question has been studied by Vysotsky and Zaporozhets~\cite{vysotsky_zaporozhets} who computed (among other results) the probability that $\pos\{S_1,\ldots,S_n\} \neq  \R^2$ for a random walk in dimension $d=2$; see also~\cite{snyder_steele} and~\cite{barndorff_nielsen_baxter} for related works.
The aim of the present paper is to derive explicit formulas for the expected values of several geometric functionals of the cones of the form $\pos\{S_1,\ldots,S_n\}$ including their expected $f$-vector and their solid angle, to mention just the simplest examples.  As it turns out, all these formulas, which we will state in Section~\ref{sec:expectations_main_results}, are distribution-free meaning that they only require certain symmetry and general position assumptions on the distribution of the increments.

One of the main observations made in~\cite{KVZ15} is a connection between positive (and convex) hulls of random walks, on the one side,  and Weyl chambers, on the other side. More specifically, random walks correspond, in a certain sense,  to Weyl chambers of type $B$, while random bridges correspond to Weyl chambers of type $A$.   This allows to restate probabilistic questions in terms of Weyl chambers and reflection arrangements. While random walks and bridges are \textit{partial sums} of random variables,  the positive hulls of the \textit{differences} of random variables, namely $\pos\{X_1-X_2,X_2-X_3,\ldots,X_{n-1}-X_n\}$ (in the $A$-case) and $\pos\{X_1-X_2,X_2-X_3,\ldots,X_{n-1}-X_n,X_n\}$ (in the $B$-case) also exhibit a distribution-free behavior and can also be interpreted in terms of Weyl chambers, as we have shown in~\cite{GK2020_WeylCones}. To summarize, we have the following types of random cones:
\begin{itemize}
\item [$(A+)$:] $\pos\{S_1,\ldots,S_n\}$, where $S_k=X_1+\ldots+X_k$, $1\leq k\leq n$, form a random bridge with $S_n=0$.
\item [$(B+)$:] $\pos\{S_1,\ldots,S_n\}$, where $S_k=X_1+\ldots+X_k$, $1\leq k\leq n$, form a random walk.
\item [$(A-)$:] $\pos\{X_1-X_2,X_2-X_3,\ldots,X_{n-1}-X_n\}$.
\item [$(B-)$:] $\pos\{X_1-X_2,X_2-X_3,\ldots,X_{n-1}-X_n, X_n\}$.
\end{itemize}
Together with any of these random cones, one may consider $3$ more random cones differing from the respective positive hull by conditioning it to be non-degenerate,  taking the dual,  or both. Altogether, we have $4$ random cones of each type.
The types $A-$ and $B-$ have been extensively studied in~\cite{GK2020_WeylCones}, while the types $A+$ and $B+$ are the main subject of the present paper. The connection between all these random cones and the Weyl chambers will be explained in more detail in Section~\ref{sec:weyl_cones}.


\subsection{Positive hulls of random walks and bridges} \label{subsec:weyl_cones_second kind}
Let us define the random cones we are interested in.


\subsubsection*{Type $A+$}
Let $X_1,\ldots,X_n$ be (possibly dependent) random vectors in $\R^d$. Consider their partial sums $S_1,\ldots,S_n$ defined by $S_k:=X_1+\ldots+X_k$ for $k\in\{1,\ldots,n\}$. Denote by $\text{Sym}(n)$ the group of all permutations of the set $\{1,\ldots,n\}$. We suppose that $X_1,\ldots,X_n$ and their partial sums satisfy the following conditions:
\begin{enumerate} [leftmargin=50pt, topsep=10pt]
\item[(Ex)]\emph{Exchangeability}: For every permutation $\sigma\in\text{Sym}(n)$, we have that
\begin{align*}
(X_{\sigma(1)},\ldots,X_{\sigma(n)})\eqdistr (X_1,\ldots,X_n).
\end{align*}
\item[(Br)]  \emph{Bridge property}: With probability $1$, it holds that $S_n=X_1+\ldots+X_n=0$.\\
\item[(GP')] \emph{General position}: $n\ge d+1$ and, for every $1\le i_1<\ldots<i_d\le n-1$, the vectors $S_{i_1},\ldots,S_{i_d}$ are linearly independent with probability $1$.
\end{enumerate}
Then, we say that $S_1,\ldots,S_n$ define a \textit{random bridge} in $\R^d$. A sufficient condition for the general position assumption (GP'), which is rather mild, will be stated  in Lemma~\ref{lemma:sufficent_condition}.
The random cone $C_n^A$ is defined as the positive hull of this random bridge:
\begin{align*}
C_n^A:=\pos\{S_1,\ldots,S_{n-1}\}.
\end{align*}

The probability of the event $\{C_n^A\neq\R^d\}$  coincides with the probability that the \textit{convex} hull of a random bridge $S_1,\dots,S_{n}$ whose increments $X_1,\dots,X_n$ satisfy (Ex), (Br) and (GP') contains $0$.  This follows from a classical result~\cite[p.~167]{Zeigler_LecturesOnPolytopes}; see also~\cite[Lemma~2.13]{GK2020_WeylCones}.   By~\cite[Theorem~2.1]{KVZ15} we have
\begin{align}
\P[C_n^A\neq\R^d]&=\P[0\notin\conv\{S_1,\ldots,S_{n-1}\}]=\frac{2}{n!}\bigg(\stirling n{d}+\stirling n{d-2}+\ldots\bigg),\label{eq:non-abs_prob_bridge}\\
\P[C_n^A=\R^d]&=\P[0\in\conv\{S_1,\ldots,S_{n-1}\}]=\frac{2}{n!}\bigg(\stirling n{d+2}+\stirling n{d+4}+\ldots\bigg),\label{eq:abs_prob_bridge}
\end{align}
where the $\stirling nk$'s are the (signless) \textit{Stirling numbers of the first kind}. By definition, $\stirling nk$ is  the number of permutations of the set $\{1,\dots,n\}$ having exactly $k$ cycles. Equivalently, the numbers $\stirling nk$  can be defined as the coefficients of the polynomial
\begin{align}\label{eq:def_stirling1_polynomial}
t(t+1)\cdot\ldots \cdot (t+n-1)=\sum_{k=0}^n\stirling{n}{k}t^k
\end{align}
for $n\in\N_0$, with the convention that $\stirling{n}{k}=0$ for $n\in \N_0$, $k\notin\{0,\ldots,n\}$, and $\stirling{0}{0} = 1$.
We remark that a consequence of~\eqref{eq:non-abs_prob_bridge} and~\eqref{eq:abs_prob_bridge} is that $\lim_{n\to\infty}\P[C_n^A=\R^d]=1$.

\subsubsection*{Type $B+$}
In the $B$-case, we assume that the random vectors $X_1,\ldots,X_n$ in $\R^d$ and their partial sums $S_1,\ldots,S_n$ satisfy the following conditions:
\begin{enumerate}[leftmargin=50pt, topsep=10pt]
\item[($\pm$Ex)]\emph{Symmetric exchangeability}: For every permutation $\sigma\in\text{Sym}(n)$ and every vector of signs $\eps=(\eps_1,\ldots,\eps_n)\in\{\pm 1\}^n$, we have that
\begin{align*}
(\eps_1X_{\sigma(1)},\ldots,\eps_nX_{\sigma(n)})\eqdistr (X_1,\ldots,X_n).
\end{align*}
\item[(GP)] \emph{General position}:  $n\ge d$ and, for every $1\le i_1<\ldots<i_d\le n$, the vectors $S_{i_1},\ldots,S_{i_d}$ are linearly independent with probability $1$.
\end{enumerate}
Under these assumptions, we refer to $S_1,\ldots,S_n$ as a symmetric \textit{random walk} in $\R^d$. Then, the random cone $C_n^B$ is defined as
\begin{align*}
C_n^B:=\pos\{S_1,\ldots,S_n\}.
\end{align*}

Again, the probability of the event $\{C_n^B\neq\R^d\}$  coincides with the so-called non-absorption probability for the convex hull of a random walk $S_1,\dots,S_n$.
By~\cite[Theorem~2.3]{KVZ15} we have
\begin{align}
\P[C_n^B\neq\R^d]&=\P[0\notin\conv\{S_1,\ldots,S_n\}]=\frac{2}{2^nn!}\sum_{r=0}^\infty \stirlingb n{d-2r-1},\label{eq:non-abs_prob_walk}\\
\P[C_n^B=\R^d]&=\P[0\in\conv\{S_1,\ldots,S_n\}]=\frac{2}{2^nn!}\sum_{r=0}^\infty \stirlingb n{d+2r+1}\label{eq:abs_prob_walk},
\end{align}
where the $\stirlingb nk$'s are the $B$-analogues to the Stirling numbers of the first kind defined  as the coefficients of the polynomial
\begin{align}\label{eq:def_stirling1b}
(t+1)(t+3)\cdot\ldots \cdot (t+2n-1)=\sum_{k=0}^n\stirlingb{n}{k}t^k
\end{align}
for $n\in\N_0$, with the convention that $\stirlingb{n}{k}=0$ for $n\in \N_0$, $k\notin\{0,\ldots,n\}$, and $\stirlingb{0}{0} = 1$. These numbers appear as entry A028338 in~\cite{sloane}. 

\subsubsection*{Dual and conditioned cones}
Together with the random cones $C_n^A$ and $C_n^B$ we also  consider their duals $(C_n^A)^\circ$ and $(C_n^B)^\circ$, as well as these four cones conditioned on the event that they are non-degenerate (that is, not equal to $\R^d$, for the original cones, or $\{0\}$, for their duals).
We will explain the construction in the $A$-case because the $B$-case is similar.
The dual cone $(C_n^A)^\circ$ of $C_n^A$ is defined by
\begin{align*}
(C_n^A)^\circ
=\{v\in\R^d:\langle v,x\rangle\le 0\text{ for all } x\in C_n^A\}
=\{v\in\R^d:\langle v,S_i\rangle\le 0 \text{ for all } i=1,\ldots,n-1\}.
\end{align*}
We denote by $\widetilde C_n^A$ the cone whose distribution is that of $C_n^A$ conditioned on the event that $C_n^A\neq\R^d$.
Similarly, its dual cone, denoted by $(\widetilde C_n^A)^\circ$ has the same distribution as $(C_n^A)^\circ$ conditioned on the event that $(C_n^A)^\circ\neq\{0\}$.
These four cones $C_n^A$, $\widetilde C_n^A$, $(C_n^A)^\circ$ and $(\widetilde C_n^A)^\circ$ are from now on referred to as the \textit{Weyl random cones of type $A+$}. The random cones $C_n^B$, $\widetilde C_n^B$, $(C_n^B)^\circ$ and $(\widetilde C_n^B)^\circ$ are defined similarly and called the \textit{Weyl random cones of type $B+$}.

\section{Main results}\label{sec:expectations_main_results}

\subsection{Geometric functionals of convex cones}\label{sec:geo_functionals}
Our goal is to compute the expectations for various geometric functionals of the random cones introduced in the previous section. Let us define the geometric functionals we are interested in. The \textit{solid angle} of a cone $C\subseteq\R^d$ is defined as
\begin{align*}
\alpha(C) := \bP[Z\in C],
\end{align*}
where $Z$ is uniformly distributed on the unit sphere inside the linear hull $\lin C$ of $C$. The set of all $k$-dimensional faces of a cone $C\subseteq \R^d$ will be denoted by $\cF_k(C)$ and the total number of $k$-faces by $f_k(C):=\#\cF_k(C)$, for all $k\in \{0,\ldots,d\}$.

We now recall two notions  generalizing
the solid angles, known as the conic intrinsic volumes and the conic quermassintegrals; see~\cite[Section 6.5]{SW08}
and~\cite{amelunxen_edge,Amelunxen2017} for more information on conic integral geometry.
On the one hand, for a cone $C\subseteq\R^d$ and $k\in\{0,\ldots,d\}$, the $k$-th \textit{conic intrinsic volume} of $C$ is defined by
\begin{align*}
\upsilon_k(C):=\sum_{F\in\cF_k(C)}\bP[(g|C)\in\relint F],
\end{align*}
where $g$ is a $d$-dimensional standard Gaussian random vector, and $g|C$ denotes the metric projection of $g$ on $C$, that is, the unique vector $y\in C$ minimizing the Euclidean distance to $g$. Finally,  $\relint F$ denotes the relative interior of $F$, i.e.\ the set of interior points of $F$ relative to its linear hull $\lin F$.

On the other hand, for $k\in\{0,\ldots,d\}$ and a cone $C\subseteq\R^d$ that is not a linear subspace, the $k$-th \textit{conic quermassintegral} is defined by
\begin{align}\label{eq:def_U}
U_k(C):=\frac 12\int_{G(d,d-k)}\1_{\{C\cap V\neq\{0\}\}}\,\nu_{d-k}(\dint V),
\end{align}
where $G(d,d-k)$ denotes the Grassmannian of all $(d-k)$-dimensional linear subspaces of $\R^d$, and $\nu_{d-k}$ denotes the uniform distribution, that is the unique normalized Haar measure, on $G(d,d-k)$. For a $j$-dimensional linear subspace $L_j\subseteq\R^d$, the conic quermassintegrals are defined in a different way, namely
\begin{align*}
U_k(L_j)=
\begin{cases}
1,		&\text{ if $j-k>0$ and odd,}\\
0,		&\text{ if $j-k\le 0$ or even.}
\end{cases}
\end{align*}
For any cone $C\subseteq \R^d$ of dimension $\dim C$, the solid angle can be recovered from either the conic intrinsic volumes or the conic quermassintegrals, via
\begin{align*}
\upsilon_{\dim C}(C)=\alpha(C)=U_{\dim C-1}(C).
\end{align*}
The so-defined functionals $U_k$ and $\upsilon_k$ are related by the following \textit{conic Crofton formula}, see~\cite[Eq.~(6.63)]{SW08} or~\cite[(2.10)]{amelunxen_edge}:
\begin{equation}\label{eq:crofton}
U_k(C) = \sum_{\substack{j=1,3,5,\ldots\\ k+j \leq d}} \upsilon_{k+j}(C).
\end{equation}

We will be interested in two series of geometric size functionals defined for polyhedral cones $C\subseteq\R^d$ and denoted by $Y_{m,l}(C)$ and $Z_{j,k}(C)$. The size functionals $Y_{m,l}$ were introduced by Hug and Schneider~\cite{HugSchneider2016} as follows:
\begin{align*}
Y_{m,l}(C):=\sum_{F\in\cF_m(C)}U_l(F), \qquad 0\le l<m\le d.
\end{align*}
For instance, for $m=\dim C$ and $l\in\{0,\ldots,d-1\}$, $Y_{m,l}(C)$ coincides with $U_l(C)$, while for $l=0$ and $m\in\{1,\ldots,d\}$, $2Y_{m,l}(C)$ coincides with the number of $m$-faces $f_m(C)$ (provided that none of the $m$-faces of $C$ are linear subspaces). Another special case of the size functionals $Y_{m,l}$ are the \textit{total $m$-face contents} $\Lambda_m(C)$ of $C$, that is,
\begin{align*}
\Lambda_m(C):=\sum_{F\in\cF_m(C)}\alpha(C)=Y_{m,m-1}(C).
\end{align*}

Let us now introduce the functionals $Z_{j,k}$. For a face $F$ of a cone $C\subseteq\R^d$, denote by $T_F(C)$ the \textit{tangent cone} of $C$ at $F$, which is defined by
\begin{align*}
T_F(C)=\{x\in\R^d:v+\eps x\in C \text{ for some $\eps>0$}\}.
\end{align*}
Here,  $v$ is any point in the relative interior $\relint F$ of $F$. Note that the dual cone of $T_F(C)$ is also called the \textit{normal cone} of $C$ at $F$ and will be denoted by $N_F(C)$. Equivalently, it can be defined by $N_F(C):=(\lin F)^\perp\cap C^\circ$. The second type of size functionals we are interested in is defined by
\begin{align*}
Z_{j,k}(C) := \sum_{F\in\cF_j(C)}U_k(T_F(C)),\qquad 0\le j\le k\le d.
\end{align*}
There is a duality relation between the $Y$- and the $Z$-functionals which will be stated in Lemma~\ref{lemma:duality_Y_Z}.

\subsection{Expected size functionals}\label{sec:gen_angle_sums}

In this section, we state our main results on the expected size functionals of the positive hulls of random walks and bridges.

\subsubsection*{Type $A+$}
Recall that the Stirling number $\stirling{n}{k}$ of the first kind is defined as the number of permutations of the set $\{1,\ldots,n\}$ having exactly $k$ cycles, or, equivalently, as the coefficient of $t^k$ in  the polynomial in~\eqref{eq:def_stirling1_polynomial}.
The \textit{Stirling number of the second kind} $\stirlingsec{n}{k}$  is defined as the number of partitions of the set $\{1,\ldots,n\}$ into $k$ non-empty subsets with the convention $\stirlingsec 00=1$.  We can now state our main results on the Weyl random cones of type $A+$.

\begin{theorem}\label{theorem:exp_size_functionals_A}
Let $S_1,\ldots,S_n$ be a random bridge in $\R^d$ whose increments $X_1,\ldots,X_n$ satisfy assumptions $\textup{(Ex)}$, $\textup{(Br)}$ and $\textup{(GP')}$. Then, for all $0\le l<m\le d-1$, it holds that
\begin{align}\label{eq:size_fct_A}
\E Y_{m,l}(C_n^A)
=\E \sum_{F\in\cF_m (C_n^A)}U_l(F)
=\frac{2}{n!}\left(\sum_{r=1}^\infty \stirling{m+1}{l+2r}\right)\left(\sum_{r=0}^\infty \stirling{n}{d-2r}\stirlingsec{d-2r}{m+1}\right).
\end{align}
\end{theorem}

\begin{theorem}\label{theorem:gen_angle_sums_quermass_A}
Let $S_1,\ldots,S_n$ be a random bridge in $\R^d$ whose increments $X_1,\ldots,X_n$ satisfy assumptions $\textup{(Ex)}$, $\textup{(Br)}$ and $\textup{(GP')}$.
Then, for all $0\le j\le k\le d$, it holds that
\begin{align*}
\E Z_{j,k}(C_n^A)
&=\E\sum_{F\in\cF_j(C_n^A)}U_k\big(T_F(C_n^A)\big)\\
&=\frac{(j+1)!}{n!}\left(\sum_{r=0}^\infty\stirling n{d-2r}\stirlingsec {d-2r}{j+1}-\sum_{r=0}^\infty\stirling n{k-2r}\stirlingsec {k-2r}{j+1}\right).
\end{align*}
\end{theorem}

The proof of Theorem~\ref{theorem:exp_size_functionals_A} is postponed to Section~\ref{sec:proof_size_functionals}, while the proof of Theorem~\ref{theorem:gen_angle_sums_quermass_A} will be presented in Section~\ref{sec:proof_gen_angle_sums}. Using the conic Crofton formula, we will derive the following two corollaries, whose proofs will be outlined in Section~\ref{sec:conic_intrinsic_vol_sums_proofs}.

\begin{corollary}\label{cor:Y_intvol_A}
Let $S_1,\ldots,S_n$ be a random bridge in $\R^d$ whose increments $X_1,\ldots,X_n$ satisfy assumptions $\textup{(Ex)}$, $\textup{(Br)}$ and $\textup{(GP')}$. For all $0\le l \leq  m\le d$, it holds that
\begin{align*}
\E\sum_{F\in\cF_m(C_n^A)}\upsilon_l(F)&= \frac{2}{n!}\stirling{m+1}{l+1}\left(\sum_{r=0}^\infty \stirling{n}{d-2r}\stirlingsec{d-2r}{m+1}\right).
\end{align*}
\end{corollary}

\begin{corollary}\label{corollary:gen_angle_sums_int_vol_A}
Let $S_1,\ldots,S_n$ be a random bridge in $\R^d$ whose increments $X_1,\ldots,X_n$ satisfy assumptions $\textup{(Ex)}$, $\textup{(Br)}$ and $\textup{(GP')}$. For all $j\in\{0,\ldots,d-1\}$, it holds that
\begin{align*}
\E\sum_{F\in\cF_j(C_n^A)}\upsilon_k\big(T_F(C_n^A)\big)&=\frac{(j+1)!}{n!}\stirling n{k+1} \stirlingsec {k+1}{j+1},\qquad k\in\{j,\ldots,d-1\},\\
\E\sum_{F\in\cF_j(C_n^A)}\upsilon_d\big(T_F(C_n^A)\big)&=\frac{(j+1)!}{n!}\sum_{r=0}^\infty(-1)^{r}\stirling n{d-r}\stirlingsec {d-r}{j+1}.
\end{align*}
\end{corollary}
As a special case, this corollary contains formulas for the expected sums of internal angles $\alpha(T_F(C_n^A))$ and external angles $\alpha(N_F(C_n^A))$ of $C_n^A$ over all $j$-dimensional faces $F\in\cF_j(C_n^A)$.


\begin{remark}\label{remark:conditioned_case_A}
The analogous expectations for the conditioned cones $\widetilde C_n^A$ instead of $C_n^A$ can be easily obtained from the above formulas. Let $\varphi$ be any non-negative functional on the set of polyhedral cones such that $\varphi(\R^d)=0$. Then, it holds that
\begin{align*}
\E\varphi (C_n^A)=\E\big[\varphi (C_n^A)\1_{\{C_n^A\neq\R^d\}}\big]+ \varphi(\R^d)\bP[C_n^A=\R^d]
=
\E\varphi(\widetilde C_n^A) \cdot \bP[C_n^A\neq\R^d].
\end{align*}
Since all functionals of $C_n^A$ from the above theorems satisfy this condition, the expectations for $\widetilde C_n^A$ follow from~\eqref{eq:non-abs_prob_bridge}. Thus, we omit the explicit formulas for the expectations of $\widetilde C_n^A$ in the above theorems and in all subsequent results from Section~\ref{sec:exp_special_cases} whose functionals satisfy this condition.
\end{remark}

\subsubsection*{Type $B+$}
Recall that the $B$-analogues to the Stirling numbers of the first kind $\stirlingb nk$ are defined as the coefficients of the polynomial in~\eqref{eq:def_stirling1b}.
The $B$-analogues to the Stirling numbers of the second kind, denoted by $\stirlingsecb nk$, are defined as
\begin{align*}
\stirlingsecb nk=\sum_{m=k}^{n}2^{m-k}\binom{n}{m}\stirlingsec{m}{k}.
\end{align*}
Note that we have $\stirlingsecb n0=1$ using the convention $\stirlingsec 00=1$.
These numbers appear as Entry A039755 in~\cite{sloane}; see also~\cite{suter}.
Note that we intentionally used square brackets in the notation of $\stirlingb nk$ and curly brackets in the notation of $\stirlingsecb nk$ to underline the analogy to the respective Stirling numbers of the first kind $\stirling nk$ and of the second kind $\stirlingsec nk$.

\begin{theorem}\label{theorem:exp_size_functionals_B}
Let $S_1,\ldots,S_n$ be a random walk in $\R^d$ whose increments $X_1,\ldots,X_n$ satisfy assumptions $\textup{($\pm$Ex)}$ and $\textup{(GP)}$. Then, for all $0\le l<m\le d-1$, it holds that
\begin{align}\label{eq:size_funct_B_unconditioned}
\E Y_{m,l}(C_n^B) =\frac{2}{2^nn!}\left(\sum_{r=0}^\infty \stirlingb{m}{l+2r+1}\right)\left(\sum_{r=0}^\infty \stirlingb{n}{d-2r-1}\stirlingsecb{d-2r-1}{m}\right).
\end{align}
\end{theorem}


\begin{theorem}\label{theorem:gen_angle_sums_quermass_B}
Let $S_1,\ldots,S_n$ be a random walk in $\R^d$ whose increments $X_1,\ldots,X_n$ satisfy assumptions $\textup{($\pm$Ex)}$ and $\textup{(GP)}$. 
Then, for all $0\le j\le k\le d$, it holds that
\begin{multline*}
\E Z_{j,k}(C_n^B)=\E\sum_{F\in\cF_j(C_n^B)}U_k\big(T_F(C_n^B)\big) \\
=\frac{j!}{2^{n-j} n!}\left(\sum_{r=0}^\infty\stirlingb n{d-2r-1}\stirlingsecb {d-2r-1}j-\sum_{r=0}^\infty\stirlingb n{k-2r-1}\stirlingsecb {k-2r-1}j\right).
\end{multline*}
\end{theorem}


These theorems are proven together with their $A$-analogues in the respective Sections~\ref{sec:proof_size_functionals} and~\ref{sec:proof_gen_angle_sums}. The conic Crofton formula yields also the following two corollaries whose proofs are postponed to Section~\ref{sec:conic_intrinsic_vol_sums_proofs}.

\begin{corollary}\label{cor:Y_intvol_B}
Let $S_1,\ldots,S_n$ be a random bridge in $\R^d$ whose increments $X_1,\ldots,X_n$ satisfy assumptions $\textup{($\pm$Ex)}$ and $\textup{(GP)}$. For all $0\le l \leq  m\le d$ we have
\begin{align*}
\E\sum_{F\in\cF_m(C_n^B)}\upsilon_l(F)&= \frac{2}{2^n n!}\stirlingb ml\left(\sum_{r=0}^\infty \stirlingb{n}{d-2r-1}\stirlingsecb{d-2r-1}{m}\right).
\end{align*}

\end{corollary}

\begin{corollary}\label{corollary:gen_angle_sums_int_vol_B}
Let $S_1,\ldots,S_n$ be a random walk in $\R^d$ whose increments $X_1,\ldots,X_n$ satisfy assumptions $\textup{($\pm$Ex)}$ and $\textup{(GP)}$. For all $j\in\{0,\ldots,d-1\}$ we have
\begin{align*}
\E\sum_{F\in\cF_j(C_n^B)}\upsilon_k\big(T_F(C_n^B)\big)&=\frac{j!}{2^{n-j}n!}\stirlingb nk \stirlingsecb kj,\quad k\in\{j,\ldots,d-1\},\\
\E\sum_{F\in\cF_j(C_n^B)}\upsilon_d\big(T_F(C_n^B)\big)&=\frac{j!}{2^{n-j}n!}\sum_{r=0}^\infty(-1)^{r}\stirlingb n{d-1-r}\stirlingsecb {d-1-r}j.
\end{align*}
\end{corollary}

\begin{remark}
In the case when the increments $X_1,\ldots,X_n$ are independent and standard Gaussian, one can show that the random cone $C_n^B$ can be viewed as a Gaussian projection of the deterministic cone  $\{x_1\geq \ldots\geq x_n\geq 0\}\subseteq \R^n$, which is a Weyl chamber of type $B$. Thus, in the Gaussian special case, Corollary~\ref{corollary:gen_angle_sums_int_vol_A} can be deduced from a formula for the intrinsic conic volume sums of the Weyl chambers~\cite[Theorem~3.3]{GK20_Schlaefli_orthoschemes} together with the invariance property of such sums under Gaussian projections~\cite[Corollary~4.18]{GKZ20_Random_Polytopes}. However, this method of proof does not extend to the non-Gaussian case.
\end{remark}


\begin{remark}\label{remark:conditioned_case_B}
The analogous expectations for $\widetilde C_n^B$ instead of $C_n^B$ are easily  obtained by dividing the above formulas by $\bP[C_n^B\neq\R^d]$ (see~\eqref{eq:non-abs_prob_walk}) since for any non-negative functional $\varphi$ on the set of polyhedral cones with $\varphi(\R^d)=0$ we have that $\E\varphi (C_n^B)= \E\varphi(\widetilde C_n^B) \cdot \bP[C_n^B\neq\R^d]$.
\end{remark}

\subsection{Special cases: A collection of expectations}\label{sec:exp_special_cases}
The formulas for the expected size functionals from the previous section comprise a lot of interesting special cases, which we state and prove below. We will present the formulas for the $A$-case and the $B$-case together.
In the subsequent results we make the following convention.  For the Weyl random cones of type $A+$, we tacitly assume that the increments $X_1,\ldots,X_n$ satisfy assumptions $\textup{(Ex)}$, $\textup{(Br)}$ and $\textup{(GP')}$, while for the Weyl random cones of type $B+$, we assume that $X_1,\ldots,X_n$ satisfy assumptions $\textup{($\pm$Ex)}$ and $\textup{(GP)}$.

\begin{corollary}\label{cor:exp_faces_A}
Let $k\in\{0,\ldots,d-1\}$ be given. Then, it holds that
\begin{align*}
\E f_k(C_n^A)
&	=\frac{2(k+1)!}{n!}\sum_{r=0}^\infty \stirling{n}{d-2r}\stirlingsec{d-2r}{k+1},\\
\E f_k(C_n^B)
&	=\frac{2\cdot k!}{2^{n-k}\cdot n!}\sum_{r=0}^\infty \stirlingb{n}{d-2r-1}\stirlingsecb{d-2r-1}{k}.
\end{align*}
\end{corollary}

\begin{remark}\label{remark:pointed}
In the proof of this corollary and other results, we will use the following observation.
Let $x_1,\ldots, x_N\in \R^d$  be vectors in general position meaning that any $d$ or less of them are linearly independent.
Then, it is known that the  positive hull $C:= \pos\{x_1,\ldots,x_N\}$
is either a linear subspace or a pointed polyhedral cone. The latter means that $C$ does not contain
a linear subspace different from $\{0\}$. This claim follows, for example, from~\cite[Eq.~(14)]{HugSchneider2016}.
In particular, this observation applies to the cones $C_n^A$ and $C_n^B$: With probability $1$, any of these cones is either equal to $\R^d$ or  is pointed.  It follows that any face of $C_n^A$ or  $C_n^B$ of dimension $k\in \{1,\ldots,d-1\}$ is a.s.~not a linear subspace.
\end{remark}
\begin{proof}[Proof of Corollary~\ref{cor:exp_faces_A}]
For $k=0$ we have $\E f_0(C_n^A) = \P[C_n^A \neq \R^d]$ and the claim follows from~\eqref{eq:non-abs_prob_bridge}.
Therefore let $k\in \{1,\ldots,d-1\}$. Replace $m$ by $k$ and $l$ by $0$ in Theorem~\ref{theorem:exp_size_functionals_A}. Then, \eqref{eq:size_fct_A} yields
\begin{align*}
\E f_k(C_n^A)
&	=2\E Y_{k,0}(C_n^A)=\frac{4}{n!}\left(\sum_{r=1}^\infty \stirling{k+1}{2r}\right)\left(\sum_{r=0}^\infty \stirling{n}{d-2r}\stirlingsec{d-2r}{k+1}\right)\\
&	=\frac{2\cdot (k+1)!}{n!}\sum_{r=0}^\infty \stirling{n}{d-2r}\stirlingsec{d-2r}{k+1}.
\end{align*}
Note that we used in the first equation that, for $k\in\{1,\ldots,d-1\}$, all $k$-faces of $C_n^A$ are not linear subspaces a.s.~(see Remark~\ref{remark:pointed}), while the last equation uses that
\begin{align*}
\stirling {k+1}1+\stirling {k+1}3+\ldots=\stirling {k+1}2+\stirling {k+1}4+\ldots=\frac{(k+1)!}2.
\end{align*}
The $B$-case follows similarly from Theorem~\ref{theorem:exp_size_functionals_B} using that
$
\stirlingb k0+\stirlingb k2+\ldots=\stirlingb k1+\stirlingb k3+\ldots=2^{k-1}k!.
$
\end{proof}

\begin{remark}
For comparison, let us mention that in~\cite{KVZ17} it has been shown that for the partial sums of the random variables $X_1,\ldots,X_n$ satisfying
$(\text{Ex})$ and $(\text{GP})$ we have
\begin{equation}\label{eq:E_F_k_C_n_main_theorem}
\E [f_k(\conv\{0,S_1,\ldots,S_n\})] = \frac{2\cdot k!}{n!} \sum_{l=0}^{\infty}\stirling{n+1}{d-2l}  \stirlingsec{d-2l}{k+1},
\quad
0\leq k\leq d-1.
\end{equation}
This formula also applies to a random bridge $S_1,\ldots, S_{n+1}$ with $S_{n+1}=0$; see~\cite[Theorem~5.1]{KVZ17}.
\end{remark}

\begin{remark}
Let us mention some asymptotic results on $\E f_k(C_n^A)$ which can be derived from Corollary~\ref{cor:exp_faces_A}. Consider first a setting where $n\to\infty$, while $d\in \N$ and $k\in \{0,\ldots,d-1\}$ stay fixed. Using the formula $\stirling{n}{d} \sim \frac{(n-1)!}{(d-1)!} (\log n)^{d-1}$ as in~\cite[Section~5.1]{KVZ15} and~\cite[Remark~1.4]{KVZ17}, we obtain
$$
\lim_{n\to\infty} \E [f_k (\widetilde C_n^A)] = \lim_{n\to\infty} \E [f_k (C_n^A)| C_n^A \neq \R^d] = (k+1)! \stirlingsec{d}{k+1},
$$
where $\widetilde C_n^A$ are the conditioned cones defined in Remark~\ref{remark:conditioned_case_A}.
The asymptotic behaviour of $\E f_k(C_n^A)$ in regimes when $n,d\to\infty$ in a coordinated way can be investigated using the methods of~\cite{kabluchko_marynych_lah} since the expression for $\E f_k(C_n^A)$ given in Corollary~\ref{cor:exp_faces_A} can be rewritten in terms of the Lah distribution introduced in~\cite{kabluchko_marynych_lah}. For example, if $k\in \N_0$ is fixed, while $d\to\infty$ and $n\to\infty$ such that $(\log n)/d \to  \gamma$ for some $\gamma\geq 0$, then Theorem~6.2 of~\cite{kabluchko_marynych_lah} implies that
$$
\lim_{n,d\to\infty} \frac{\E f_k(C_n^A)}{\binom {n-1}{k}} =
\begin{cases}
1, &\text{ if } \gamma < 1/(k+1),\\
0, &\text{ if } \gamma > 1/(k+1).
\end{cases}
$$
\end{remark}

\begin{corollary}\label{cor:quermass}
Let $k\in\{0,\ldots,d\}$ be given. Then, it holds that
\begin{align*}
\E U_k(C_n^A)
	&=\begin{dcases}
	\frac{1}{n!}\bigg(\sum_{r=1}^\infty\stirling n{k+2r}+\sum_{r=1}^\infty\stirling n{d+2r}\bigg),	& d-k\text { odd},\\
	\frac{1}{n!}\bigg(\stirling{n}{k+2}+\stirling{n}{k+4}+\ldots+\stirling{n}{d}\bigg),	& d-k=0\text { or even},
	\end{dcases}\\
\E U_k(C_n^B)
	&=\begin{dcases}
	\frac{1}{2^nn!}\bigg(\sum_{r=0}^\infty\stirlingb n{k+2r+1}+\sum_{r=0}^\infty\stirlingb n{d+2r+1}\bigg),	& d-k\text { odd},\\
	\frac{1}{2^nn!}\big(\stirlingb{n}{k+1}+\stirlingb{n}{k+3}+\ldots+\stirlingb{n}{d-1}\big),	& d-k=0\text { or even},
	\end{dcases}
\end{align*}
while in the conditioned case we have
\begin{align*}
\E U_k(\widetilde C_n^A)
&	=\frac{\sum_{r=0}^\infty\stirling n{d-2r}- \sum_{r=0}^\infty\stirling n {k-2r}}{2\sum_{r=0}^\infty \stirling n{d-2r}},\\
\E U_k(\widetilde C_n^B)
&	=\frac{\sum_{r=0}^\infty\stirlingb n{d-2r-1}- \sum_{r=0}^\infty\stirlingb n {k-2r-1}}{2\sum_{r=0}^\infty \stirlingb n{d-2r-1}}.
\end{align*}
\end{corollary}

\begin{proof}
The $A$-case follows from Theorem~\ref{theorem:gen_angle_sums_quermass_A} in the special case $j=0$. In particular, we know that $C_n^A$ is pointed if and only if $C_n^A\neq\R^d$. This yields
\begin{align*}
\E\sum_{F\in\cF_0(C_n^A)}U_k\big(T_F(C_n^A)\big)=\E U_k\big(T_{\{0\}}(C_n^A)\big)\1_{\{C_n^A\neq\R^d\}}=\E U_k(C_n^A)\1_{\{C_n^A\neq\R^d\}},
\end{align*}
which leads us to
\begin{align*}
\E U_k(C_n^A)=\E\sum_{F\in\cF_0(C_n^A)}U_k\big(T_F(C_n^A)\big)+U_k(\R^d)\bP\big[C_n^A=\R^d\big].
\end{align*}
Applying Theorem~\ref{theorem:gen_angle_sums_quermass_A} together with the formula for the absorption probability~\eqref{eq:abs_prob_bridge} and the fact that $U_k(\R^d)=1$ if $d-k>0$ and odd (and otherwise $U_k(\R^d)=0$) proves the first claim. The formula for $\widetilde C_n^A$ follows from Theorem~\ref{theorem:gen_angle_sums_quermass_A}  in the special case $j=0$ together with Remark~\ref{remark:conditioned_case_A}.
The $B$-case is proven analogously using Theorem~\ref{theorem:gen_angle_sums_quermass_B}.
\end{proof}

%

\begin{corollary}\label{cor:intr_vol_A}
For the expected conic intrinsic volumes, it holds that
\begin{align*}
\E \upsilon_d(C_n^A)&=\frac{1}{n!}\sum_{r=1}^\infty \stirling n{d+r},\qquad
\E \upsilon_k(C_n^A)=\frac{1}{n!}\stirling n{k+1},\quad k\in\{0,\ldots,d-1\},\\
\E \upsilon_d(C_n^B)&=\frac{1}{2^nn!}\sum_{r=0}^\infty\stirlingb n{d+r},
\qquad
\E \upsilon_k(C_n^B)
	= \frac{1}{2^nn!}\stirlingb{n}{k},\quad k\in\{0,\ldots,d-1\}.
\end{align*}
In the conditioned case,  we have
\begin{align*}
\E \upsilon_k(\widetilde C_n^A)&=\frac{\stirling n{k+1}}{2\sum_{r=0}^\infty\stirling n {d-2r}},
\quad
\E \upsilon_k(\widetilde C_n^B)
	=\frac{\stirlingb{n}{k}}{2\sum_{r=0}^\infty \stirlingb{n}{d-2r-1}},\quad k\in\{0,\ldots,d-1\},\\
\E \upsilon_d(\widetilde C_n^A)&=\frac{\sum_{r=0}^\infty (-1)^r\stirling n{d-r}}{2\sum_{r=0}^\infty\stirling n{d-2r}},
\quad
\E \upsilon_d(\widetilde C_n^B)=\frac{\sum_{r=0}^{\infty}(-1)^r\stirlingb n{d-r-1}}{2\sum_{r=0}^\infty \stirlingb{n}{d-2r-1}}.
\end{align*}
\end{corollary}

\begin{proof}
We consider only  the $A$-case. Essentially, the claims follow from Corollary~\ref{corollary:gen_angle_sums_int_vol_A} in the case $j=0$. Indeed, in the same way as in the proof of Corollary~\ref{cor:quermass}, we observe that
\begin{align*}
\E\upsilon_k(C_n^A)=\E\sum_{F\in\cF_0(C_n^A)}\upsilon_k\big(T_F(C_n^A)\big)+\upsilon_k(\R^d)\bP\big[C_n^A=\R^d\big]=\E\sum_{F\in\cF_0(C_n^A)}\upsilon_k\big(T_F(C_n^A)\big),
\end{align*}
for $k\in\{0,\ldots,d-1\}$.
Thus, applying Corollary~\ref{corollary:gen_angle_sums_int_vol_A} yields the formulas for $\E \upsilon_k(C_n^A)$ in the case $k\in\{0,\ldots,d-1\}$ . The case $k=d$ follows from the relation $\upsilon_0(C)+\ldots+\upsilon_d(C)=1$, which holds for any cone $C\subseteq\R^d$. The formulas for $\widetilde C_n^A$ follow from  Remark~\ref{remark:conditioned_case_A}. The $B$-case is proven analogously using Corollary~\ref{corollary:gen_angle_sums_int_vol_B}.
\end{proof}

\begin{corollary}
Let $k\in\{1,\ldots,d-1\}$ be given. Then, it holds that
\begin{align*}
\E \Lambda_k(C_n^A)&	=\frac{2}{n!}\sum_{r=0}^\infty \stirling{n}{d-2r}\stirlingsec{d-2r}{k+1},\quad\E \Lambda_k(C_n^B)	=\frac{2}{2^nn!}\sum_{r=0}^\infty \stirlingb n{d-2r-1}\stirlingsecb {d-2r-1}k.
\end{align*}
\end{corollary}

\begin{proof}
This follows directly from Theorems~\ref{theorem:exp_size_functionals_A} and~\ref{theorem:exp_size_functionals_B} by replacing $m$ by $k$ and $l$ by $k-1$, and observing that $\stirling kk=1= B[k,k]$.
\end{proof}




\subsubsection*{The dual cones}

Similar formulas can be derived for the dual cones $(C_n^A)^\circ$, $(\widetilde C_n^A)^\circ$ in the $A$-case and $(C_n^B)^\circ$, $(\widetilde C_n^B)^\circ$ in the $B$-case. 
For example, formulas for the expected number of $k$-faces and the expected conic intrinsic volumes of the dual cones $(C_n^A)^\circ$ and $(C_n^B)^\circ$ follow directly from Corollaries~\ref{cor:exp_faces_A}  and~\ref{cor:intr_vol_A} using the duality relations $f_k(C)=f_{d-k}(C^\circ)$ and $\upsilon_k(C)=\upsilon_{d-k}(C^\circ)$ which hold for every cone $C\subseteq\R^d$. The expectations of the more general functionals $Y_{m,l}$ and $Z_{j,k}$ can be derived by means of the following duality relation.

\begin{lemma}\label{lemma:duality_Y_Z}
For any cone $C\subseteq\R^d$ with non-empty interior, and all $0\le l<m\le d$, it holds that
\begin{align}\label{eq:relation_Y_Z}
Y_{m,l}(C^\circ)=\frac{1}{2}f_{d-m}(C) - Z_{d-m,d-l}(C).
\end{align}
\end{lemma}
\begin{proof}
Take a $(d-m)$-face $G\in\cF_{d-m}(C)$.
By definition, the normal cone of $G$ with respect to $C$ is given by $N_G(C)=(\lin G)^\perp\cap C^\circ$. Its dual cone, the tangent cone $T_G(C)$, takes the form
\begin{align}\label{eq:tangent_cone_decomposition}
T_G(C)=\big((\lin G)^\perp\cap C^\circ\big)^\circ=\lin G+C=\lin G + \big(C|(\lin G)^\perp\big),
\end{align}
where we used the relation $(C\cap D)^\circ=C^\circ+D^\circ$ which holds for any two polyhedral cones $C,D\subseteq\R^d$. Let us also argue that $T_G(C)$ is not a linear subspace. Indeed,  we can find a supporting hyperplane $H$ of $C$ such that $H\cap C=G$ and $C$ is contained in one of the closed half-spaces, let's say $H^-$, bounded by $H$. Thus, $\lin G\subseteq H$, which implies that $C|(\lin G)^\perp$ is contained in the half-space $H^-\cap (\lin G)^\perp$ inside $(\lin G)^\perp$. Together with the fact that $C$ has non-empty interior in $\R^d$ (and therefore also $C|(\lin G)^\perp$ has non-empty interior inside $(\lin G)^\perp$), this yields that $T_G(C)$, and equivalently $N_G(C)$, is not a linear subspace.
This allows us to apply the duality relation $U_l(N_G(C)) + U_{d-l}(T_G(C))=\frac 12$, see~\cite[Eq.~(5)]{HugSchneider2016}, which yields
\begin{align*}
Y_{m,l}\big(C^\circ\big)
&=\sum_{F\in\cF_{m}(C^\circ)}U_l(F)
	=\sum_{G\in\cF_{d-m}(C)}U_l\big(N_G(C)\big)\\
&	=\sum_{G\in\cF_{d-m}(C)}\Big(\frac 12-U_{d-l}\big(T_G(C))\Big)
		=\frac{1}{2} f_{d-m}(C) - Z_{d-m,d-l}(C).
\end{align*}
The  proof of~\eqref{eq:relation_Y_Z} is complete.
\end{proof}
Using Lemma~\ref{lemma:duality_Y_Z} together with our results on $Y_{m,l}$ and $Z_{j,k}$, it is possible to compute the expected values of these functionals for the dual cones of $C_n^A$ and $C_n^B$. We state only the result for $Y_{m,l}$.
\begin{corollary}\label{corollary:size_funct_dual_A}
Let $0\le l<m\le d$ be given. Then, it holds that
\begin{align*}
\E Y_{m,l}\big((C_n^A)^\circ\big)
&	=\frac{(d-m+1)!}{n!}\sum_{r=0}^\infty\stirling n{d-l-2r}\stirlingsec{d-l-2r}{d-m+1},\\
\E Y_{m,l}\big((C_n^B)^\circ\big)
&	=\frac{(d-m)!}{2^{n-d+m}n!}\sum_{r=0}^\infty\stirlingb n{d-l-2r-1}\stirlingsecb{d-l-2r-1}{d-m}.
\end{align*}
\end{corollary}

\begin{proof}
The claims follow from Lemma~\ref{lemma:duality_Y_Z} (with $C=C_n^A$ and $C=C_n^B$) by inserting the formulas from Corollary~\ref{cor:exp_faces_A} (with $k$ replaced by $d-m\in\{0,\ldots,d-1\}$) and Theorems~\ref{theorem:gen_angle_sums_quermass_A} and~\ref{theorem:gen_angle_sums_quermass_B} (with $j$ replaced by $d-m$ and $k$ replaced by $d-l$).
\end{proof}

The formulas for $\E U_k((C_n^A)^\circ)$ and $\E U_k((C_n^B)^\circ)$ follow from Corollary~\ref{corollary:size_funct_dual_A} in the special case $m=d$ and $l=k$, while the formulas for $\E \Lambda_k((C_n^A)^\circ)$ and $\E \Lambda_k((C_n^B)^\circ)$ can be obtained by considering the special case $m=k$ and $l=k-1$.

\section{Comparison to other random cones}\label{sec:weyl_cones}

In this section, we explain the connection between the Weyl random cones of types $A+$ and $B+$  and the Weyl chambers of types $A_{n-1}$ and  $B_n$. Additionally, we recall another model of random cones from~\cite{GK2020_WeylCones}, which we call  Weyl random cones of types $A-$ and $B-$, and show how both models are related.
Additionally, we compare these constructions to the Cover-Efron cones~\cite{CoverEfron,HugSchneider2016} which are related to the Weyl chambers of product type $B_1\times \ldots\times B_1$.

\subsection{Weyl random cones generated by differences} \label{subsec:weyl_cones_first_kind}
We start by recalling the construction of random cones introduced in~\cite{GK2020_WeylCones}.
Let $X_1,\ldots,X_n$ be (possibly dependent) random vectors in $\R^d$.

\subsubsection*{Type $A-$}
In the $A$-case, we suppose that $X_1,\ldots,X_n$ satisfy condition $\textup{(Ex)}$, and additionally, the following general position assumption holds:
\begin{enumerate}[leftmargin=50pt, topsep=10pt]
\item[(GP*)] $n\ge d+1$ and for every $\sigma\in\text{Sym}(n)$  any $d$ vectors from the list
$
X_{\sigma(1)}-X_{\sigma(2)},X_{\sigma(2)}-X_{\sigma(3)},\ldots,X_{\sigma(n-1)}-X_{\sigma(n)}
$
are linearly independent in $\R^d$, with probability $1$.
\end{enumerate}
Here, $\text{Sym}(n)$ denotes the group of all permutations of the set $\{1,\ldots,n\}$. Consider the random conical tessellation $\mathcal{W}^A(X_1,\ldots,X_n)$ of $\R^d$ generated by the hyperplane arrangement
\begin{align*}
\big\{(X_i-X_j)^\perp:\:1\le i<j\le n\big\}.
\end{align*}
Here, $v^\bot$ is the orthogonal complement of a vector $v\in\R^d\backslash\{0\}$.
The above hyperplanes dissect the space $\R^d$ into open, connected components whose closures define the polyhedral cones of the conical tessellation. This tessellation a.s~consists of
$
2\sum_{r=0}^\infty\stirling n{n-d+2r+1}
$
$d$-dimensional cones; see~\cite[Theorem~1.1]{GK2020_WeylCones}.

The random cone $\widetilde D_n^A$ (introduced in~\cite{GK2020_WeylCones} and denoted there by $\mathcal D_n^A$) is defined as follows: Among the cones of the random tessellation $\mathcal{W}^A(X_1,\ldots,X_n)$ we pick one uniformly at random. Equivalently, $\widetilde D_n^A$  has the same distribution as  the random cone
\begin{align}\label{eq:weyl2_unconditioned}
D_n^A:=\{v\in\R^d:\langle v,X_1\rangle\le\ldots\le\langle v,X_n\rangle\}
\end{align}
conditioned on the event that $D_n^A\neq\{0\}$; see~\cite[Proposition~1.12]{GK2020_WeylCones}. The dual cone $(\widetilde D_n^A)^\circ$
has the same distribution as the random cone
\begin{align}\label{eq:weyl2dual_unconditioned}
(D_n^A)^\circ=\pos\{X_1-X_2,\ldots,X_{n-1}-X_n\}
\end{align}
conditioned on the event $(D_n^A)^\circ\neq\R^d$; see~\cite[Proposition~1.13]{GK2020_WeylCones}.
It is also possible to consider the random cones from~\eqref{eq:weyl2_unconditioned} and~\eqref{eq:weyl2dual_unconditioned} without the conditioning. This yields four random cones that differ only by conditioning on a certain event, or by taking the dual cone. These four random cones are referred to as the \textit{Weyl random cones of type $A-$}.

\subsubsection*{Type $B-$}
In the $B$-case, we suppose that $X_1,\ldots,X_n$ satisfy $\textup{($\pm$Ex)}$ and the following general position assumption holds:
\begin{enumerate}[leftmargin=50pt, topsep=10pt]
\item[(GP**)] $n\ge d$ and for every $\sigma\in\text{Sym}(n)$ and $\varepsilon=(\eps_1,\ldots,\eps_n)\in\{\pm 1\}^n$  any $d$ vectors of the list
$$
\varepsilon_1X_{\sigma(1)}-\varepsilon_2X_{\sigma(2)},\varepsilon_2X_{\sigma(2)}-\varepsilon_3X_{\sigma(3)},\ldots,\varepsilon_{n-1}X_{\sigma(n-1)}-\varepsilon_nX_{\sigma(n)}, \varepsilon_nX_{\sigma(n)}
$$
are linearly independent in $\R^d$, with probability $1$.
\end{enumerate}
Now, consider the random conical tessellation $\mathcal{W}^B(X_1,\ldots,X_n)$  of $\R^d$ generated by the hyperplane arrangement
\begin{align*}
\big\{(X_i-X_j)^\perp,(X_i+X_j)^\perp:\:1\le i<j\le n\big\}\cup
\big\{X_i^\perp:\:1\le i\le n\big\}.
\end{align*}
In~\cite[Theorem~1.3]{GK2020_WeylCones} it was shown that the tessellation $\mathcal{W}^B(X_1,\ldots,X_n)$ almost surely consists of
$
2\sum_{r=0}^\infty\stirlingb n{n-d+2r+1}
$
$d$-dimensional cones.
Then, the random cone $\widetilde D_n^B$ (denoted in~\cite{GK2020_WeylCones} by $\mathcal D_n^B$) is defined as follows: Among the cones of the random tessellation $\mathcal{W}^B(X_1,\ldots,X_n)$ pick one uniformly at random. Like in the $A$-case, it has the same distribution as the random cone
\begin{align}\label{eq:weyl1_unconditioned}
D_n^B:=\{v\in\R^d:\langle v,X_1\rangle\le\ldots\le\langle v,X_n\rangle\le 0\}
\end{align}
conditioned on the event that $D_n^B\neq\{0\}$. The dual cone $(\widetilde D_n^B)^\circ$ of $\widetilde D_n^B$ has the same distribution as the random cone
\begin{align}\label{eq:weyl1dual_unconditioned}
(D_n^B)^\circ=\pos\{X_1-X_2,\ldots,X_{n-1}-X_n,X_n\},
\end{align}
conditioned on the event this $(D_n^B)^\circ\neq\R^d$; see~\cite[Proposition~1.14]{GK2020_WeylCones}.
This yields four random cones $D_n^B, \widetilde D_n^B, (D_n^B)^\circ, (\widetilde D_n^B)^\circ$ that differ only by conditioning or duality. We refer to these four cones as the \textit{Weyl random cones of type $B-$}.

\subsection{Connections between the Weyl cones and the Weyl chambers}\label{subsec:conection_first_second_kind}
In order to see the similarity in the construction of the Weyl random cones generated by sums ($A+$ and $B+$) and those generated by differences ($A-$ and $B-$) we shortly introduce the Weyl chambers of type $A_{n-1}$ and  $B_n$. The polyhedral cone
$$
W_A:=\{x\in\R^n:x_1\ge x_2\ge \ldots\ge x_n\}
$$
is called the \textit{fundamental Weyl chamber} (or just Weyl chamber) of type $A_{n-1}$. Applying to $W_A$ arbitrary permutations of the coordinates yields the following polyhedral cones which form a conical tessellation of $\R^n$:
\begin{align*}
\{x\in\R^n:x_{\sigma(1)}\ge x_{\sigma(2)}\ge \ldots\ge x_{\sigma(n)}\},\quad\sigma\in\text{Sym}(n).
\end{align*}
These cones are also called Weyl chambers of type $A_{n-1}$.
Similarly, the cone
\begin{align*}
W_B:=\{x\in\R^n:x_1\ge x_2\ge \ldots\ge x_n\ge 0\}
\end{align*}
is called the \textit{fundamental Weyl chamber} (or just Weyl chamber) of type $B_n$. The cones obtained from $W_B$ by arbitrarily permuting the coordinates and multiplying any number of coordinates by $-1$ form a conical tessellation of $\R^n$. These cones are also called Weyl chambers of type $B_n$ and are explicitly given by
\begin{align*}
\{x\in\R^n:\eps_1x_{\sigma(1)}\ge \eps_2x_{\sigma(2)}\ge \ldots\ge \eps_nx_{\sigma(n)}\ge 0\},\quad\sigma\in\text{Sym}(n), \,\eps=(\eps_1,\ldots,\eps_n)\in\{\pm 1\}^n.
\end{align*}
%

We now claim that the random Weyl cones can be obtained from the Weyl chambers and their duals by applying linear transformations.
Let the random $(n\times d)$-matrix $X=(X_1,\ldots,X_n)$ consist of the random column vectors $X_1,\ldots,X_n$ introduced above. Let $e_1,\dots,e_n$ denote the standard Euclidean orthonormal basis vectors in $\R^n$.
In the $A$-case, the fundamental Weyl chamber and its negative dual cone can equivalently be written in terms of positive hulls as follows:
$$
W_A=\pos\{e_1,e_1+e_2,\ldots,e_1+\ldots+e_n,-(e_1+\ldots+e_n)\}
$$
and
$$
(-W_A)^\circ
=\pos\{e_1-e_2,\ldots,e_{n-1}-e_n\}.
$$
Then, the cone $C_n^A$ can be represented as
\begin{align*}
C_n^A=\pos\{S_1,S_2,\ldots,S_{n-1}\}=X\cdot W_A,
\end{align*}
which relies on the bridge property (Br), whereas for $(D_n^A)^\circ$ we have
\begin{align*}
(D_n^A)^\circ= \pos\{X_1-X_2,\ldots,X_{n-1}-X_n\}=X\cdot (-W_A)^\circ.
\end{align*}

Similarly, the Weyl chamber of type $B_n$ and its negative dual can be represented in terms of positive hulls as follows:
$$
W_B=\pos\{e_1,e_1+e_2,\ldots,e_1+\ldots+e_n\}
$$
and
$$
(-W_B)^\circ
=\pos\{e_1-e_2,\ldots,e_{n-1}-e_n,e_n\}.
$$
Then, the cone $C_n^B$ can be represented as
\begin{align*}
C_n^B=\pos\{S_1,S_2,\ldots,S_{n}\}=X\cdot W_B,
\end{align*}
while for $(D_n^B)^\circ$  we have
\begin{align*}
(D_n^B)^\circ=\pos\{X_1-X_2,\ldots,X_{n-1}-X_n,X_n\}=X\cdot (-W_B)^\circ.
\end{align*}

\begin{remark}
Comparing these constructions to the construction of the Cover-Efron random cone, see~\cite{CoverEfron} and~\cite{HugSchneider2016}, defined as $\pos\{X_1,\ldots,X_n\}$ conditioned on the event that this positive hull is not equal to $\R^d$, we observe that there are again four cones belonging to the Cover-Efron family that differ only by duality or conditioning. Namely, the so-called Schl\"afli random cone, see~\cite{HugSchneider2016}, is defined as the cone drawn uniformly at random from the conical tessellation generated by $X_1^\perp,\ldots,X_n^\perp$.
The Schl\"afli random cone can equivalently be described as the cone $\{v\in\R^d:\langle v,X_i\rangle\le 0\;\forall i=1,\ldots,n\}$ conditioned on the event that it is not $\{0\}$. Removing the conditioning operation  yields  the cone
$\pos\{X_1,\ldots,X_n\}$, which was among others studied by Donoho and Tanner~\cite{Donoho2009}, and its dual cone $\{v\in\R^d:\langle v,X_i\rangle\le 0\;\forall i=1,\ldots,n\}$.
Similarly to Weyl cones, we can observe that
\begin{align*}
\pos\{X_1,\ldots,X_n\}=X\cdot V = X\cdot (-V)^\circ,
\end{align*}
where the matrix $X$ is defined as above and $V:=\{x\in\R^n:x_i\geq 0\;\forall i=1,\ldots,n\}$ is the non-negative orthant. Unlike in the Weyl case, we have that $V=(-V)^\circ$ and therefore there is no difference between the plus and minus types in the Cover-Efron case.
\end{remark}

\begin{remark}
Contrary to what one may assume at first glance, the conditions (GP') and (GP*) (respectively (GP) and (GP**)) are not equivalent. We explain this in the case of (GP') and (GP*). By~\cite[Theorem~2.16]{GK2020_WeylCones}, (GP*) holds if and only if the kernel $\ker X$ has dimension $n-d$ and, additionally, $(\ker X)^\perp$ is in general position with respect to the reflection arrangement $\cA(A_{n-1}):=\{(e_i-e_j)^\perp:1\le i<j\le n\}$.
By~\cite[Lemma~6.2]{KVZ15} (which states just one direction, the other one being analogous), condition (GP') holds if and only if $\ker X$ has dimension $n-d$ and, additionally, $\ker X$ (not $(\ker X)^\perp$!) is in general position with respect to $\cA(A_{n-1})$.

We now convince ourselves that these conditions are not equivalent by looking at  the case $d=1$.
Indeed,  a linear hull $\lin v$ of a vector $v\in\R^n\backslash\{0\}$ is in general position with respect to $\cA(A_{n-1})$ if and only if the coordinates of $v$ are pairwise distinct.
On the other hand, the linear hyperplane $H:= v^{\bot}$ is in general position with respect to $\cA(A_{n-1})$ if and only if there is no partition of the coordinates $(v_1,\dots,v_n)$ of $v$ such that the sum of coordinates inside each subset of the partition vanishes.
Clearly, these conditions are not equivalent.
\end{remark}

\section{Proofs}\label{sec:proofs}

This section contains the proofs of our main results stated in Section~\ref{sec:expectations_main_results}. In Section~\ref{sec:prob_positive_hulls} we recall a formula for the absorption probability of a joint convex hull of random walks and random bridges, and, additionally, prove some consequences of this result that are needed for the subsequent proofs. Section~\ref{sec:proof_size_functionals} contains the proofs of Theorems~\ref{theorem:exp_size_functionals_A} and~\ref{theorem:exp_size_functionals_B} on the expected size functionals $Y_{m,l}$ of $C_n^A$ and $C_n^B$, while in Section~\ref{sec:proof_gen_angle_sums} we prove the formulas for expected values of $Z_{j,k}$ of $C_n^A$ and $C_n^B$, that is, Theorems~\ref{theorem:gen_angle_sums_quermass_A} and~\ref{theorem:gen_angle_sums_quermass_B}. Corollaries~\ref{corollary:gen_angle_sums_int_vol_A} and~\ref{corollary:gen_angle_sums_int_vol_B} will be derived in Section~\ref{sec:conic_intrinsic_vol_sums_proofs}.
Finally, Section~\ref{sec:proof_sufficient_condition} contains Lemma~\ref{lemma:sufficent_condition} which states a sufficient condition for the general position assumptions (GP) and (GP'). Whenever possible, we treat the $B$-case of each result more extensively, and only briefly sketch the $A$-case, since it is proven in a similar and somewhat simpler way.

\subsection{Probabilities involving joint positive hulls of random walks and random bridges}\label{sec:prob_positive_hulls}

Before we will be able to prove the results from Section~\ref{sec:expectations_main_results} we need to recall a formula for the absorption probability of the joint convex hull of a finite number of random walks and random bridges that was proven in~\cite[Theorem~2.1]{KVZ17} and generalizes  Equations~\eqref{eq:abs_prob_bridge} and~\eqref{eq:abs_prob_walk}.   We will state this result, which is the basis of the subsequent proofs, in Theorem~\ref{theorem:absorpt_prob_joint_convex hull}. But first, we need to introduce some necessary notation.

Fix numbers $s,r\in\N_0$ that do not vanish simultaneously, and take $n_1,\ldots,n_s\in\N$, $m_1,\ldots,m_r\in\N\backslash\{1\}$. Then, consider some random vectors
\begin{align*}
X_1^{(1)},\ldots,X_{n_1}^{(1)},\;\ldots,\;X_1^{(s)},\ldots,X_{n_s}^{(s)},
\qquad
\;Y_1^{(1)},\ldots,Y_{m_1}^{(1)},\;\ldots,\;Y_1^{(r)},\ldots,Y_{m_r}^{(r)}
\end{align*}
in $\R^d$, such that $Y_1^{(j)}+\ldots+Y_{m_j}^{(j)}=0$ a.s.~for every $1\le j\le r$.  Now, consider the collection of $s$ random walks $(S_1^{(i)},\ldots,S_{n_i}^{(i)})$, $1\le i\le s$, and $r$ random bridges $(R_1^{(j)},\ldots,R_{m_j}^{(j)})$, $1\le j\le r$, defined by
$$
S_k^{(i)} := X_{1}^{(i)} + \ldots + X_{k}^{(i)},
\qquad
R_\ell^{(j)} := Y_{1}^{(j)} + \ldots + Y_{\ell}^{(j)}.
$$
Let their joint convex hull be denoted by
\begin{align}\label{eq:joint_convex_hull_walks_bridges}
H:=\conv\Big\{S_1^{(1)},\ldots,S_{n_1}^{(1)},\;\ldots,\;S_1^{(s)},\ldots,S_{n_s}^{(s)},R_1^{(1)},\ldots,R_{m_1-1}^{(1)},\;\ldots,\;R_1^{(r)},\ldots,R_{m_r-1}^{(r)}\Big\}
\end{align}
and their joint positive hull by
\begin{align}\label{eq:joint_pos_hull_walks_bridges}
P:=\pos\Big\{S_1^{(1)},\ldots,S_{n_1}^{(1)},\;\ldots,\;S_1^{(s)},\ldots,S_{n_s}^{(s)},R_1^{(1)},\ldots,R_{m_1-1}^{(1)},\;\ldots,\;R_1^{(r)},\ldots,R_{m_r-1}^{(r)}\Big\}.
\end{align}

We are interested in the probability that $H$ contains the origin. It is known explicitly under the following distributional invariance assumption:
\begin{multline}\label{eq:exchangeability_joint_walks_bridges}
\Big(X_1^{(1)},\ldots,X_{n_1}^{(1)},\;\ldots,\;X_1^{(s)},\ldots,X_{n_s}^{(s)},\;Y_1^{(1)},\ldots,Y_{m_1}^{(1)},\;\ldots,\;Y_1^{(r)},\ldots,Y_{m_r}^{(r)}\Big)	\\
	\eqdistr \Big(\eps_1^{(1)}X_{\sigma_1(1)}^{(1)},\ldots,\eps_{n_1}^{(1)}X_{\sigma_1(n_1)}^{(1)},\;\ldots,\;\eps_1^{(s)}X_{\sigma_s(1)}^{(s)},\ldots,\eps_{n_s}^{(s)}X_{{\sigma_s(n_s)}}^{(s)},\\
	Y_ {\pi_1(1)}^{(1)},\ldots,Y_{{\pi_1(m_1)}}^{(1)},\;\ldots,\;Y_{\pi_r(1)}^{(r)},\ldots,Y_{{\pi_r(m_r)}}^{(r)}\Big)	,
\end{multline}
for all permutations $\sigma^{(1)}\in\text{Sym}(n_1),\ldots,\sigma^{(s)}\in\text{Sym}(n_s),\pi^{(1)}\in\text{Sym}(m_1),\ldots,\pi^{(r)}\in\text{Sym}(m_r)$ and all signs $\eps_1^{(1)},\ldots,\eps_{n_1}^{(1)},\ldots,\eps_1^{(s)},\ldots,\eps_{n_s}^{(s)}\in \{\pm 1\}$. Essentially, \eqref{eq:exchangeability_joint_walks_bridges} states that the joint distribution of the increments of the walks/bridges is invariant with respect to the product action of reflection groups of type $B$ (for walks) and $A$ (for bridges).
\begin{theorem}[Theorem~2.1 in~\cite{KVZ17}]\label{theorem:absorpt_prob_joint_convex hull}
Suppose that~\eqref{eq:exchangeability_joint_walks_bridges} holds and that any $d$ random vectors from the list on the right-hand side of~\eqref{eq:joint_convex_hull_walks_bridges} are linearly independent with probability $1$. Then, it holds that
\begin{align*}
\P[0\in H]=\frac{2(P(d+1)+P(d+3)+\ldots)}{2^{n_1}n_1!\cdot\ldots\cdot 2^{n_s}n_s!m_1!\cdot\ldots\cdot m_r!},\quad \P[0\notin H]=\frac{2(P(d-1)+P(d-3)+\ldots)}{2^{n_1}n_1!\cdot\ldots\cdot 2^{n_s}n_s!m_1!\cdot\ldots\cdot m_r!},
\end{align*}
where the $P(k)$'s are the coefficients of the polynomial
\begin{align*}
\left(\prod_{i=1}^s(t+1)(t+3)\ldots(t+2n_i-1)\right)\left(\prod_{j=1}^r(t+1)(t+2)\ldots(t+m_j-1)\right)=\sum_{k=0}^\infty P(k)t^k.
\end{align*}
Note that the $P(k)$'s also depend on $s,r,n_1,\ldots,n_s,m_1,\ldots,m_r$.
\end{theorem}


We will now state and prove several consequences of Theorem~\ref{theorem:absorpt_prob_joint_convex hull}. We say that two linear subspaces $L,L'\subseteq \R^d$ are in \textit{general position} if $\dim (L\cap L')=\max\{0,\dim L+\dim L'-d\}$. The next corollary gives a formula for the probability that a linear subspace (which satisfies certain general position assumptions) intersects the joint positive hull $P$ from~\eqref{eq:joint_pos_hull_walks_bridges} non-trivially.

\begin{corollary}\label{cor:prob_walkbridge_intersects_V}
Suppose that~\eqref{eq:exchangeability_joint_walks_bridges} holds and that any $d$ random vectors from the list on the right-hand side of~\eqref{eq:joint_pos_hull_walks_bridges} are linearly independent with probability $1$.
Let $k\in\{0,\ldots,d-1\}$ and $V\in G(d,d-k)$ be a deterministic $(d-k)$-dimensional linear subspace of $\R^d$ which is a.s.\ in general position to each linear hull $\lin\{Z_{1},\ldots,Z_{k}\}$, where $Z_{1},\ldots,Z_{k}$ are any $k$ vectors from the list on the right hand-side of~\eqref{eq:joint_pos_hull_walks_bridges}. Then, it holds that
\begin{align*}
\bP[P\cap V\neq\{0\}]=\frac{2(P(k+1)+P(k+3)+\ldots)}{2^{n_1}n_1!\cdot\ldots\cdot 2^{n_s}n_s!m_1!\cdot\ldots\cdot m_r!},
\end{align*}
where the $P(k)$'s are the coefficients defined in Theorem~\ref{theorem:absorpt_prob_joint_convex hull}.
\end{corollary}

\begin{remark}\label{remark:prob_walkbridge_intersects_V}
Note that in the case $s=1$, $r=0$, (i.e.\ for one random walk of length $n=n_1$) we obtain
\begin{align*}
\bP[C_n^B\cap V\neq\{0\}]=\P[\pos\{S_1,\ldots,S_n\}\cap V\neq \{0\}]=\frac{2}{2^nn!}\sum_{r=0}^\infty \stirlingb{n}{k+2r+1}.
\end{align*}
Similarly, for $s=0$ and $r=1$ (i.e.\ for one random bridge of length $n=m_1$) we get
\begin{align*}
\bP[C_n^A\cap V\neq\{0\}]=\P[\pos\{S_1,\ldots,S_{n-1}\}\cap V\neq \{0\}]=\frac{2}{n!}\sum_{r=1}^\infty \stirling{n}{k+2r}.
\end{align*}
\end{remark}

\begin{proof}[Proof of Corollary~\ref{cor:prob_walkbridge_intersects_V}]
To keep the notation simple, we only prove the case when  $P$ is the joint positive hull of $s=1$ random walk, denoted by $S_1,\ldots,S_n$, and $r=1$ random bridge, denoted by $R_1,\ldots,R_m$, from which the general case is evident.
Fix $k\in\{0,\ldots,d-1\}$. Our aim is to show that
\begin{equation}\label{eq:pos_huelle_wird_geschnitten}
\pos\{S_1,\ldots,S_n,R_1,\ldots,R_{m-1}\}\cap V\neq \{0\}
\end{equation}
if and only if
\begin{equation}\label{eq:projektion_enthaelt_null}
0\in\conv\{S_1|V^\perp,\ldots,S_n|V^\perp,R_1|V^\perp,\ldots,R_{m-1}|V^\perp\}.
\end{equation}
After that, we can apply the formula for the absorption probability from Theorem~\ref{theorem:absorpt_prob_joint_convex hull} in the ambient linear subspace $V^\perp$ to the projected random walk $S_1|V^\perp,\ldots,S_n|V^\perp$ and the projected random bridge $R_1|V^\perp,\ldots,R_{m}|V^\perp$. Here, $x|V^\bot$ denotes the orthogonal projection of a vector $x\in\R^d$ to $V^\bot$.

Suppose that $\pos\{S_1,\ldots,S_n,R_1,\ldots,R_{m-1}\}\cap V\neq \{0\}$, or equivalently, there is a point $x\in V\bsl \{0\}$ and numbers $\lambda_1,\ldots,\lambda_{n+m-1}\ge 0$ that do not vanish simultaneously such that
\begin{align*}
\lambda_1S_1+\ldots+\lambda_nS_n+\lambda_{n+1}R_1+\ldots+\lambda_{n+m-1}R_{m-1}=x.
\end{align*}
By projecting on $V^\perp$, this implies that there are numbers $\lambda_1,\ldots,\lambda_{n+m-1}\ge 0$ that do not vanish simultaneously such that
\begin{align}\label{eq:sum_0}
\lambda_1(S_1|V^\perp)+\ldots+\lambda_n(S_n|V^\perp)+\lambda_{n+1}(R_1|V^\perp)+\ldots+\lambda_{n+m-1}(R_{m-1}|V^\perp)=0.
\end{align}
Dividing by the sum of the $\lambda_i$'s proves that~\eqref{eq:projektion_enthaelt_null} holds.
To prove the inverse implication, suppose that there are $(\lambda_1,\ldots,\lambda_{n+m-1})\in\R^{n+m-1}_{\ge 0}\backslash \{0\}$ such that $\lambda_1S_1|V^\perp+\ldots+\lambda_{n+m-1}R_{m-1}|V^\perp=0$. Using the linearity of the projection map  $x\mapsto x|V^\perp$, this implies that $\lambda_1S_1+\ldots+\lambda_nS_n+\lambda_{n+1}R_1+\ldots+\lambda_{n+m-1}R_{m-1}\in V$. If this vector is not $0$, then~\eqref{eq:pos_huelle_wird_geschnitten} holds.  If it is $0$, the linear independence assumption of Corollary~\ref{cor:prob_walkbridge_intersects_V} implies that
\begin{align*}
\pos\{S_1,\ldots,S_n,R_1,\ldots,R_{m-1}\}=\R^d,
\end{align*}
see, e.g., Lemma~2.13 from~\cite{GK2020_WeylCones}, and, again, ~\eqref{eq:pos_huelle_wird_geschnitten} holds.  This proves that~\eqref{eq:pos_huelle_wird_geschnitten} and~\eqref{eq:projektion_enthaelt_null} are equivalent.

Note that the invariance assumption~\eqref{eq:exchangeability_joint_walks_bridges}  easily translates to the increments of the projected random walks and bridges
\begin{equation}\label{eq:proj_walk_bridges}
S_1|V^\perp,\ldots,S_n|V^\perp,R_1|V^\perp,\ldots,R_{m-1}|V^\perp.
\end{equation}
To be able to apply Theorem~\ref{theorem:absorpt_prob_joint_convex hull}
in the ambient $k$-dimensional linear subspace $V^\perp$, we need to prove that each $k$ of the projected vectors~\eqref{eq:proj_walk_bridges}
are linearly independent with probability $1$.  
To this end, we take any $k$ vectors, denoted by $Z_1,\ldots,Z_k$, from the list $S_1,\ldots,S_n,R_1,\ldots,R_{m-1}$, assume that their projections $Z_1|V^\perp,\ldots,Z_k|V^\perp$ are linearly dependent and show that this event has probability $0$.
Without loss of generality,  let
\begin{align*}
Z_1|V^\perp=\lambda_2(Z_2|V^\perp)+\ldots+\lambda_k(Z_k|V^\perp)=(\lambda_2Z_2+\ldots+\lambda_kZ_k)|V^\perp
\end{align*}
for some numbers $\lambda_2,\ldots,\lambda_k\in\R$. This already implies that $Z_1-\lambda_2Z_2-\ldots-\lambda_kZ_k$ belongs to the linear subspace $V$. Since $Z_1-\lambda_2Z_2-\ldots-\lambda_kZ_k\neq 0$ (due to their linear independence), this means that
$$
\lin\{Z_1,\ldots,Z_k\}\cap V\neq\{0\}.
$$
However, this event has probability $0$ since $\dim\lin\{Z_1,\ldots,Z_k\}=k$  and $V$ is assumed to be in general position with respect to $\lin\{Z_1,\ldots,Z_k\}$, with probability $1$. This completes the proof.
\end{proof}

Now, we want to evaluate the so-called face probability, i.e.\ the probability that, for given indices $1\le i_1<\ldots<i_k\le n$, the cone  $\pos\{S_{i_1},\ldots,S_{i_k}\}$ is a $k$-face of the positive hull $C_n^B$ of a random walk $S_1,\ldots,S_n$ (and, similarly, for bridges).   The basic idea is that the projection of $C_n^B$ onto the linear subspace $M:=\lin^\perp\{S_{i_1},\ldots,S_{i_k}\}$ yields a number of random walks and bridges inside $M$. Furthermore, the face probability coincides with the non-absorption probability of the joint convex hull of these random walks and bridges in $M$ which can be evaluated by means of  Theorem~\ref{theorem:absorpt_prob_joint_convex hull}. For convex hulls of random walks instead of positive hulls, the face probability has already been evaluated in~\cite[Theorems~1.6, 1.11]{KVZ17}. 

\begin{remark}\label{remark:faces_simplicial}
Let $k\in\{1,\ldots,d-1\}$. We claim that under the assumption (GP) on $X_1,\ldots,X_n$, each $k$-face $F$ of $C_n^B$ is of the form
$$
F=\pos\{S_{i_1},\ldots,S_{i_k}\}
$$
for some indices $1\le i_1<\ldots<i_k\le n$. Indeed,  it is known that $F$ is of the form $F=\pos\{S_{i_1},\ldots,S_{i_m}\}$ for some $m\in\{k,\ldots,n\}$ and $1\le i_1<\ldots<i_m\le n$. Assumption (GP) then implies that $\dim F=\min\{m,d\}$. But since $\dim F=k$ this already implies that $k=m$. 

Similarly, under the assumptions (Br) and (GP') on $X_1,\ldots,X_n$, each $k$-face $G$ of $C_n^A$ is of the form
$$
G=\pos\{S_{i_1},\ldots,S_{i_k}\}
$$
for some indices $1\le i_1<\ldots<i_k\le n-1$.
\end{remark}

\begin{corollary}\label{cor:face_probabilities_pos_walk}
Let $S_1,\ldots,S_n$ be a random walk in $\R^d$ whose increments $X_1,\ldots,X_n$ satisfy assumption $\textup{(GP)}$. Fix some $k\in\{1,\ldots,d-1\}$ and let $1\le i_1<\ldots<i_k\le n$ be any indices. Suppose that the distribution of $(X_1,\ldots,X_n)$ stays invariant under arbitrary simultaneous permutations of the elements inside the blocks $(X_1,\ldots,X_{i_1}),(X_{i_1+1},\ldots,X_{i_2}),\ldots,(X_{i_{k-1}+1},\ldots,X_{i_k})$ and   signed permutations inside the block $(X_{i_k+1},\ldots,X_n)$. Then, it holds that
\begin{align}
\bP[\pos(S_{i_1},\ldots,S_{i_k})\in\cF_k(C_n^B)]&=2\sum_{r=0}^{\lfloor(d-k-1)/2\rfloor}\frac{P_{i_1,i_2-i_1\ldots,i_k-i_{k-1}} ^{(n)}(d-k-2r-1)}{i_1!(i_2-i_1)!\cdot\ldots\cdot(i_k-i_{k-1})!(n-i_k)! 2^{n-i_k}},\label{eq:face_probab1}\\
\bP[\pos(S_{i_1},\ldots,S_{i_k})\notin\cF_k(C_n^B)]&=2\sum_{r=0}^{\lfloor(n-d-1)/2\rfloor}\frac{P_{i_1,i_2-i_1\ldots,i_k-i_{k-1}} ^{(n)}(d-k+2r+1)}{i_1!(i_2-i_1)!\cdot\ldots\cdot(i_k-i_{k-1})!(n-i_k)! 2^{n-i_k}},\label{eq:face_probab2}
\end{align}
where the coefficients $P_{j_1,\ldots,j_k} ^{(n)}(r)$ are defined by
\begin{multline}\label{eq:coefficients_joint_convex_hull}
(t+1)(t+3)\cdot\ldots\cdot(t+2(n-(j_1+\ldots+j_k))-1)\\\times\prod_{l=1}^k(t+1)(t+2)\cdot\ldots\cdot(t+j_l-1)=\sum_{r=0}^{n-k}P_{j_1,\ldots,j_k}^{(n)}(r)t^r.
\end{multline}
By convention, we set $P_{j_1,\ldots,j_k}^{(n)}(r)=0$ for $r<0$ or $r>n-k$.
\end{corollary}

\begin{proof}
It is easy to check that, for indices $k\in\{1,\ldots,d-1\}$ and $1\le i_1<\ldots<i_k\le n$, we have the following equivalence:
\begin{align*}
\pos\{S_{i_1},\ldots,S_{i_k}\}\in\cF_k(C_n^B)\;\;\Leftrightarrow\;\; \conv\{0,S_{i_1},\ldots,S_{i_k}\}\in\cF_k(\conv\{0,S_1,\ldots,S_n\}).
\end{align*}
Under slightly less general invariance assumptions on $X_1,\ldots,X_n$, the probability of the latter event was computed in~\cite[Theorem~1.6]{KVZ17}. However, it is easy to see that the invariance assumptions we imposed in the present corollary are sufficient for the proof of Theorem~1.6 of~\cite{KVZ17} to be valid. This proves~\eqref{eq:face_probab1}. To prove~\eqref{eq:face_probab2} we need to check that  the right-hand sides of~\eqref{eq:face_probab1} and~\eqref{eq:face_probab2} sum to  $1$. This can easily be  seen by taking $t= \pm 1$ in~\eqref{eq:coefficients_joint_convex_hull}.
\end{proof}

\begin{corollary}\label{cor:face_probabilities_pos_bridge}
Let $S_1,\ldots,S_n$ be a random bridge in $\R^d$ whose increments $X_1,\ldots,X_n$ satisfy assumptions $\textup{(Br)}$ and $\textup{(GP')}$. Fix some $k\in \{1,\ldots,d-1\}$ and some indices $1\le i_1<\ldots<i_k\le n-1$.  Also, suppose that the distribution of the random vector $(X_1,\ldots,X_n)$ stays invariant under arbitrary simultaneous permutations of the elements inside the blocks $(X_1,\ldots,X_{i_1}),\ldots,(X_{i_{k-1}+1},\ldots,X_{i_k})$, $(X_{i_{k}+1},\ldots,X_{n})$.  Then, it holds that
\begin{align*}
\bP[\pos(S_{i_1},\ldots,S_{i_k})\in\cF_k(C_n^A)]&=2\sum_{r=0}^{\lfloor(d-k-1)/2\rfloor}\frac{ Q_{i_1,i_2-i_1\ldots,i_k-i_{k-1}} ^{(n)}(d-k-2r-1)}{i_1!(i_2-i_1)!\cdot\ldots\cdot(i_k-i_{k-1})!(n-i_k)! },\\
\bP[\pos(S_{i_1},\ldots,S_{i_k})\notin\cF_k(C_n^A)]&=2\sum_{r=0}^{\lfloor(n-d-2)/2\rfloor}\frac{ Q_{i_1,i_2-i_1\ldots,i_k-i_{k-1}} ^{(n)}(d-k+2r+1)}{i_1!(i_2-i_1)!\cdot\ldots\cdot(i_k-i_{k-1})!(n-i_k)! },
\end{align*}
where the coefficients $Q_{j_1,\ldots,j_k} ^{(n)}(r)$ are defined by
\begin{align}\label{eq:coefficients_joint_convex_hull_bridge}
\prod_{l=1}^{k+1}(t+1)(t+2)\cdot\ldots\cdot(t+j_l-1)=\sum_{r=0}^{n-k-1}Q_{j_1,\ldots,j_k}^{(n)}(r)t^r.
\end{align}
By convention, we set $j_{k+1}:=n-(j_1+\ldots+j_k)$, and also $Q_{j_1,\ldots,j_k}^{(n)}(r)=0$ for $r<0$ or $r>n-k-1$.
\end{corollary}

\begin{proof}
This follows from Theorem~1.11 of~\cite{KVZ17} (replacing $(i_1,\ldots,i_{k+1})$ by $(0,i_1,\ldots,i_k)$) with the same argumentation as in the proof of Corollary~\ref{cor:face_probabilities_pos_walk}.
\end{proof}

\subsection{Expected size functionals \texorpdfstring{$Y_{m,l}$}{Y\_\{m,l\}}: Proof of Theorems~\ref{theorem:exp_size_functionals_A} and~\ref{theorem:exp_size_functionals_B}}\label{sec:proof_size_functionals}

Before proving the theorems, we need another lemma  which states that the event that
$\pos\{S_{i_1},\ldots,S_{i_m}\}$ is a face of $C_n^B$ is independent of any functional of this face.

\begin{lemma}\label{lemma:independence}
Let $S_1,\ldots,S_n$ be a random walk whose increments $X_1,\ldots,X_n$ satisfy the assumptions $\textup{($\pm$Ex)}$ and $\textup{(GP)}$. Fix $1\le m\le d-1$ and indices $1\le i_1<\ldots<i_m\le n$ and take a measurable function $\varphi:\cC^d\to[0,\infty]$, where $\cC^d$ denotes the space of all polyhedral cones in $\R^d$. Then, it holds that
\begin{multline*}
\bE\big[\varphi(\pos\{S_{i_1},\ldots,S_{i_m}\})\ind_{\{\pos\{S_{i_1},\ldots,S_{i_m}\}\in\cF_m(C_n^B)\}}\big]\\
 =\bE\big[\varphi(\pos\{S_{i_1},\ldots,S_{i_m}\})\big]\cdot\P\big[\pos\{S_{i_1},\ldots,S_{i_m}\}\in\cF_m(C_n^B)\big].
\end{multline*}
\end{lemma}

\begin{proof}
Let $\varsigma_1,\ldots,\varsigma_{m+1}$ be uniformly distributed permutations of the sets $\{1,\ldots,i_1\},\ldots,\{i_{m-1}+1,\ldots,i_m\},\{i_m+1,\ldots,n\}$, respectively. Furthermore, let $\epsilon$ be a uniformly distributed vector of signs on $\{\pm 1\}^{n-i_m}$. By extending the probability space, we can assume each random permutation and the vector of signs to be independent of each other and of the increments $X_1,\ldots,X_n$. We then define $S^*_1,\ldots,S^*_n$ to be the random walk with increments
\begin{align}\label{eq:signed_permut_increments}
X_{\varsigma_1(1)},\ldots,X_{\varsigma_1(i_1)},
\;\;\ldots,\;\;
X_{\varsigma_{m}(i_{m-1}+1)},\ldots,X_{\varsigma_{m}(i_m)},
\epsilon_1X_{\varsigma_{m+1}(i_m+1)},\ldots,\epsilon_{n-i_m}X_{\varsigma_{m+1}(n)}.
\end{align}
We set $(C_n^B)^*:=\pos\{S^*_1,\ldots,S^*_n\}$. Using assumption ($\pm$Ex) and the independence of the random permutations and random signs from the random walk $S_1,\ldots,S_n$, we obtain that
\begin{multline}\label{eq:exp_equal}
\bE\big[\varphi(\pos\{S_{i_1},\ldots,S_{i_m}\})\ind_{\{\pos\{S_{i_1},\ldots,S_{i_m}\}\in\cF_m(C_n^B)\}}\big]\\
=	\bE\big[\varphi(\pos\{S^*_{i_1},\ldots,S^*_{i_m}\})\ind_{\{\pos\{S^*_{i_1},\ldots,S^*_{i_m}\}\in\cF_m((C_n^B)^*)\}}\big].
\end{multline}
Let $(x_1,\ldots,x_n)$ be some realization of $(X_1,\ldots,X_n)$. The corresponding realization of the partial sums $(S_1,\ldots,S_n)$ is denoted by $(s_1,\ldots,s_n)$.
Now, we condition on the event $(X_1,\ldots,X_n) = (x_1,\ldots,x_n)$, while still allowing $\varsigma_1,\ldots,\varsigma_{m+1}$ and $\epsilon$ to be random. The corresponding values of $(S_1^*,\ldots,S_n^*)$ are denoted by $(s_1^*,\ldots,s_n^*)$. By using the independence  and the fact that $s_{i_1}^*=s_{i_1},\ldots,s_{i_m}^*=s_{i_m}$, we obtain
\begin{align*}
&\bE\big[\varphi(\pos\{S^*_{i_1},\ldots,S^*_{i_m}\})\ind_{\{\pos\{S^*_{i_1},\ldots,S^*_{i_m}\}\in\cF_m((C_n^B)^*)\}}\big]\\
&	\quad=\int_{(\R^d)^n}\E\big[\varphi(\pos\{S^*_{i_1},\ldots,S^*_{i_m}\})\ind_{\{\pos\{S^*_{i_1},\ldots,S^*_{i_m}\}\in\cF_m((C_n^B)^*)\}}\big|X_1=x_1,\ldots,X_n=x_n\big]\\
&	\qquad\qquad\qquad\times\bP_{(X_1,\ldots,X_n)}(\dint(x_1,\ldots,x_n))\\
&	\quad=\int_{(\R^d)^n} \varphi(\pos\{s_{i_1},\ldots,s_{i_m}\})\P\big[\pos\{s^*_{i_1},\ldots,s^*_{i_m}\}\in\cF_m\big((c_n^B)^*\big)\big]\,\bP_{(X_1,\ldots,X_n)}(\dint(x_1,\ldots,x_n)).
\end{align*}
Here, we defined $(c_n^B)^*:=\pos\{s^*_1,\ldots,s^*_n\}$. Observe that $s^*_1,\ldots,s^*_n$ are random variables because they depend on the random permutations and random signs introduced above.
As constructed, the random walk $s^*_1,\ldots,s^*_n$ satisfies the invariance assumption of Corollary~\ref{cor:face_probabilities_pos_walk} and also the  general position assumption (GP) for $\bP_{(X_1,\ldots,X_n)}$-almost every  vector $(x_1,\ldots,x_n)$. Thus, we can apply Corollary~\ref{cor:face_probabilities_pos_walk} to the probability under the integral. Observe that this probability does not depend on $(x_1,\ldots,x_n)$.  Taking~\eqref{eq:exp_equal} into account we obtain
\begin{align*}
&\bE\big[\varphi(\pos\{S_{i_1},\ldots,S_{i_m}\})\ind_{\{\pos\{S_{i_1},\ldots,S_{i_m}\}\in\cF_m(C_n^B)\}}\big]\\
&	\quad= \frac{2 \sum_{r=0}^{\lfloor(d-m-1)/2\rfloor} P_{i_1,i_2-i_1,\ldots,i_m-i_{m-1}} ^{(n)}(d-m-2r - 1)}{i_1!(i_2-i_1)!\cdot\ldots\cdot(i_m-i_{m-1})!(n-i_m)! 2^{n-i_m}}\cdot\bE\big[\varphi(\pos\{S_{i_1},\ldots,S_{i_m}\})\big]\\
&	\quad=\bE\big[\varphi(\pos\{S_{i_1},\ldots,S_{i_m}\})\big]\bP\big[\pos\{S_{i_1},\ldots,S_{i_m}\}\in\cF_m(C_n^B)\big],
\end{align*}
which completes the proof.
\end{proof}

The analogous result in the $A$-case is  the following lemma. It is proven in the same way using Corollary~\ref{cor:face_probabilities_pos_bridge}.

\begin{lemma}\label{lemma:independence_bridge}
Let $S_1,\ldots,S_n$ be a random bridge whose increments $X_1,\ldots,X_n$ satisfy the assumptions $\textup{(Ex)}$, $\textup{(Br)}$ and $\textup{(GP')}$. Fix $1\le m\le d-1$ and indices $1\le i_1<\ldots<i_m\le n-1$ and take a measurable function $\varphi:\cC^d\to[0,\infty]$, where $\cC^d$ denotes the space of all polyhedral cones in $\R^d$. Then, it holds that
\begin{multline*}
\bE\big[\varphi(\pos\{S_{i_1},\ldots,S_{i_m}\})\ind_{\{\pos\{S_{i_1},\ldots,S_{i_m}\}\in\cF_m(C_n^A)\}}\big]\\
 =\bE\big[\varphi(\pos\{S_{i_1},\ldots,S_{i_m}\})\big]\cdot\P\big[\pos\{S_{i_1},\ldots,S_{i_m}\}\in\cF_m(C_n^A)\big].
\end{multline*}
\end{lemma}

Now, we are finally able to prove Theorem~\ref{theorem:exp_size_functionals_B}.

\begin{proof}[Proof of Theorem~\ref{theorem:exp_size_functionals_B}]
Let $0\le l<m\le d-1$ be given. It holds that
\begin{align}\label{eq:sum1}
\E Y_{m,l}(C_n^B)
&	=\E\sum_{F\in\cF_m(C_n^B)}U_l(F)\notag\\
&	=\sum_{1\le i_1<\ldots<i_m\le n}\E\left[U_l\big(\pos\{S_{i_1},\ldots,S_{i_m}\}\big)\1_{\{\pos\{S_{i_1},\ldots,S_{i_m}\}\in\cF_m(C_n^B)\}}\right],
\end{align}
where we used Remark~\ref{remark:faces_simplicial} in the last step. Because of (GP),
$\pos\{S_{i_1},\ldots,S_{i_m}\}$  is  not a linear subspace. Thus, we can use the definition of $U_l$ given in~\eqref{eq:def_U} and Fubini's theorem to rewrite the sum on the right-hand side of~\eqref{eq:sum1} as follows:
\begin{align*}
\int_{G(d,d-l)}\frac{1}{2}\sum_{1\le i_1<\ldots<i_m\le n}\P\big[\pos\{S_{i_1},\ldots,S_{i_m}\}\cap V\neq\{0\},\pos\{S_{i_1},\ldots,S_{i_m}\}\in\cF_m(C_n^B)\big]\,\nu_{d-l}(\dint V).
\end{align*}
In order to prove~\eqref{eq:size_funct_B_unconditioned}, we show that for $\nu_{d-l}$-a.e.~$V\in G(d,d-l)$, it holds that
\begin{align}\label{eq:to_prove_size_funct}
&\sum_{1\le i_1<\ldots<i_m\le n}\P\big[\pos\{S_{i_1},\ldots,S_{i_m}\}\cap V\neq\{0\},\pos\{S_{i_1},\ldots,S_{i_m}\}\in\cF_m(C_n^B)\big]\notag\\
&	\quad=\frac{4}{2^nn!}\left(\sum_{r=0}^\infty \stirlingb{m}{l+2r+1}\right)\left(\sum_{r=0}^\infty \stirlingb{n}{d-2r-1}\stirlingsecb{d-2r-1}{m}\right).
\end{align}
To this end, we use Lemma~\ref{lemma:independence} with $\varphi(C):=\1\{C\cap V\neq\{0\}\}$ to obtain
\begin{align}\label{eq:sum2}
&\sum_{1\le i_1<\ldots<i_m\le n}\P\big[\pos\{S_{i_1},\ldots,S_{i_m}\}\cap V\neq\{0\},\pos\{S_{i_1},\ldots,S_{i_m}\}\in\cF_m(C_n^B)\big]\notag\\
&	\quad=\sum_{1\le i_1<\ldots<i_m\le n}\P\big[\pos\{S_{i_1},\ldots,S_{i_m}\}\cap V\neq\{0\}\big]\P\big[\pos\{S_{i_1},\ldots,S_{i_m}\}\in\cF_m(C_n^B)\big]\notag\\
&	\quad=\sum_{\substack{j_1,\ldots,j_m\in\N:\\j_1+\ldots+j_m\le n}}\bP\big[\pos\{S_{j_1},\ldots,S_{j_1+\ldots+j_m}\}\cap V\neq\{0\}\big]\bP\big[\pos\{S_{j_1},\ldots,S_{j_1+\ldots+j_m}\}\in\cF_m(C_n^B)\big],
\end{align}
where used the transformation of indices $j_1:=i_1,j_2:=i_2-i_1,\ldots,j_m:=i_m-i_{m-1}$ in the last equation. Now, we want to apply Corollary~\ref{cor:face_probabilities_pos_walk} to the second probability, and Remark~\ref{remark:prob_walkbridge_intersects_V} to the first probability. Since the latter is not directly applicable we need to invest some further work. Therefore, fix a permutation $\sigma\in\text{Sym}(m)$.
For each tuple $(j_1,\ldots,j_m)\in\N^m$ satisfying $j_1+\ldots+j_m\le n$, the tuple $(j_{\sigma(1)},\ldots,j_{\sigma(m)})$ also satisfies $j_{\sigma(1)}+\ldots+j_{\sigma(m)}\le n$.
Thus, the sum in~\eqref{eq:sum2} does not change if we replace $(j_1,\ldots,j_m)$ by $(j_{\sigma(1)},\ldots,j_{\sigma(m)})$ inside the sum. It follows that the sum in~\eqref{eq:sum2} can be rewritten as follows:
\begin{align}\label{eq:sum3}
&\frac{1}{m!}\sum_{\sigma\in\text{Sym}(m)}\sum_{\substack{j_1,\ldots,j_m\in\N:\\j_1+\ldots+j_m\le n}}\bP\big[\pos\{S_{j_{\sigma(1)}},\ldots,S_{j_{\sigma(1)}+\ldots+j_{\sigma(m)}}\}\cap V\neq\{0\}\big]\notag\\
&\qquad\qquad\times\bP\big[\pos\{S_{j_{\sigma(1)}},\ldots,S_{j_{\sigma(1)}+\ldots+j_{\sigma(m)}}\}\in\cF_m(C_n^B)\big].
\end{align}
We claim that the last probability takes the same value for all permutations $\sigma\in\text{Sym}(n)$. Indeed, Corollary~\ref{cor:face_probabilities_pos_walk} yields that
\begin{align*}
&\bP\big[\pos\{S_{j_{\sigma(1)}},\ldots,S_{j_{\sigma(1)}+\ldots+j_{\sigma(m)}}\}\in\cF_m(C_n^B)\big]\\
&	\quad=\frac{2(P_{j_{\sigma(1)},j_{\sigma(2)},\ldots,j_{\sigma(m)}} ^{(n)}(d-m-1)+P_{j_{\sigma(1)},j_{\sigma(2)},\ldots,j_{\sigma(m)}} ^{(n)}(d-m-3)+\ldots)}{j_{\sigma(1)}!j_{\sigma(2)}!\cdot\ldots\cdot j_{\sigma(m)}!(n-(j_1+\ldots+j_m))! 2^{n-(j_1+\ldots+j_m)}}\\
&	\quad=\frac{2(P_{j_1,j_2,\ldots,j_m} ^{(n)}(d-m-1)+P_{j_1,j_2,\ldots,j_m} ^{(n)}(d-m-3)+\ldots)}{j_1!j_2!\cdot\ldots\cdot j_m!(n-(j_1+\ldots+j_m))! 2^{n-(j_1+\ldots+j_m)}}\\
&	\quad=\bP\big[\pos\{S_{j_1},\ldots,S_{j_1+\ldots+j_m}\}\in\cF_m(C_n^B)\big],
\end{align*}
which follows from the definition the coefficients $P_{j_1,\ldots,j_k}^{(n)}(r)$ in~\eqref{eq:coefficients_joint_convex_hull}.
Furthermore, we introduce the random variables
$$
T_1:=S_{j_1},T_2:=S_{j_1+j_2}-S_{j_1},\ldots,T_m:=S_{j_1+\ldots+j_{m}}-S_{j_1+\ldots+j_{m-1}}
$$
that are the increments of $S_{j_1},S_{j_1+j_2},\ldots,S_{j_1+\ldots+j_m}$. Additionally, let $\varsigma$ be a  random and uniform permutation of the set $\{1,\ldots,m\}$ which is independent of the random walk $S_1,\ldots,S_n$. Using the exchangeability of $X_1,\ldots,X_n$, we can write
\begin{align*}
&\frac{1}{m!}\sum_{\sigma\in\text{Sym}(m)}\bP\big[\pos\{S_{j_{\sigma(1)}},\ldots,S_{j_{\sigma(1)}+\ldots+j_{\sigma(m)}}\}\cap V\neq\{0\}\big]\\
&	\quad=\sum_{\sigma\in\text{Sym}(m)}\frac{\bP[\pos\{T_{\sigma(1)},T_{\sigma(1)}+T_{\sigma(2)},\ldots,T_{\sigma(1)}+\ldots+T_{\sigma(m)}\}\cap V\neq\{0\}]}{m!}\\
&	\quad=\bP\big[\pos\{T_{\varsigma(1)},T_{\varsigma(1)}+T_{\varsigma(2)},\ldots,T_{\varsigma(1)}+\ldots+T_{\varsigma(m)}\}\cap V\neq \{0\}\big].
\end{align*}
The increments $T_{\varsigma(1)},\ldots,T_{\varsigma(m)}$ are symmetrically exchangeable by construction and their
partial sums are a.s.~linearly independent due to (GP). It is well known~\cite[Lemma~13.2.1]{SW08} that  $\nu_{d-l}$-a.e~$V\in G(d,d-l)$ is in general position to any deterministic linear subspace. It follows that we can apply Corollary~\ref{cor:prob_walkbridge_intersects_V}, or rather Remark~\ref{remark:prob_walkbridge_intersects_V}, to the increments $T_{\varsigma(1)},\ldots,T_{\varsigma(m)}$ instead of $X_1,\ldots,X_n$ and obtain
\begin{align*}
\bP\big[\pos\{T_{\varsigma(1)},T_{\varsigma(1)}+T_{\varsigma(2)},\ldots,T_{\varsigma(1)}+\ldots+T_{\varsigma(m)}\}\cap V\neq \{0\}\big]=\frac{2}{2^mm!}\sum_{r=0}^\infty\stirlingb{m}{l+2r+1}.
\end{align*}
Taking all these results into consideration, we arrive at
\begin{align*}
&\sum_{1\le i_1<\ldots<i_m\le n}\P\big[\pos\{S_{i_1},\ldots,S_{i_m}\}\cap V\neq\{0\},\pos\{S_{i_1},\ldots,S_{i_m}\}\in\cF_m(C_n^B)\big]\\
&	\quad=\sum_{\substack{j_1,\ldots,j_m\in\N:\\j_1+\ldots+j_m\le n}}\bP\big[\pos\{S_{j_1},\ldots,S_{j_1+\ldots+j_m}\}\in\cF_m(C_n^B)\big]\frac{2}{2^mm!}\sum_{r=0}^\infty\stirlingb{m}{l+2r+1}\\
&	\quad=\frac{2}{2^mm!}\sum_{r=0}^\infty\stirlingb{m}{l+2r+1}\sum_{\substack{j_1,\ldots,j_m\in\N:\\j_1+\ldots+j_{m}\le n}}\frac{2(P_{j_1,j_2,\ldots,j_m} ^{(n)}(d-m-1)+P_{j_1,j_2,\ldots,j_m} ^{(n)}(d-m-3)+\ldots)}{j_1!j_2!\cdot\ldots\cdot j_{m}!(n-(j_1+\ldots+j_m))! 2^{n-(j_1+\ldots+j_m)}}\\
&	\quad=\left(\frac{4}{2^mm!}\sum_{r=0}^\infty\stirlingb{m}{l+2r+1}\right)\Bigg(\sum_{r=0}^\infty\: \sum_{\substack{j_1,\ldots,j_m\in\N,j_{m+1}\in\N_0:\\j_1+\ldots+j_{m+1}= n}}\frac{P_{j_1,j_2,\ldots,j_m} ^{(n)}(d-m-2r-1)}{j_1!j_2!\cdot\ldots\cdot j_{m+1}! 2^{j_{m+1}}}\Bigg).
\end{align*}
To prove~\eqref{eq:to_prove_size_funct}, which would yield~\eqref{eq:size_funct_B_unconditioned}, we need to verify that
\begin{align}\label{eq:to_prove2_size_funct}
\sum_{\substack{j_1,\ldots,j_m\in\N,j_{m+1}\in\N_0:\\j_1+\ldots+j_{m+1}= n}}\frac{P_{j_1,j_2,\ldots,j_m} ^{(n)}(d-m-2r-1)}{j_1!j_2!\cdot\ldots\cdot j_{m+1}! 2^{j_{m+1}}}=\frac{m!}{2^{n-m}n!}\stirlingb{n}{d-2r-1}\stirlingsecb{d-2r-1}{m}
\end{align}
holds for all admissible $r$.
This is done similarly to the proof of Theorem~1.2 in~\cite{KVZ17}. Using the notation $t^{\overline{j}}:=t(t+1)\cdot\ldots(t+j-1)$ for the rising factorial, and  $[t^N]f(t)=\frac{1}{N!}f^{(N)}(0)$ for the coefficient of $t^N$ in the Taylor expansion of a function $f$ around $0$, we obtain
\begin{align*}
&\sum_{\substack{j_1,\ldots,j_m\in\N,j_{m+1}\in\N_0:\\j_1+\ldots+j_{m+1}= n}}\frac{P_{j_1,j_2,\ldots,j_m} ^{(n)}(d-m-2r-1)}{j_1!j_2!\cdot\ldots\cdot j_{m+1}! 2^{j_{m+1}}}\notag\\
&\quad=	\big[t^{d-m-2r-1}\big]\sum_{\substack{j_1,\ldots,j_m\in\N,j_{m+1}\in\N_0:\\j_1+\ldots+j_{m+1}= n}}\left(\frac{(t+1)(t+3)\cdot\ldots\cdot(t+2j_{m+1}-1)}{2^{j_{m+1}}j_{m+1}!}\cdot\frac{t^{\overline{j_1}}}{t\cdot j_1!}\cdot\ldots\cdot\frac{t^{\overline{j_m}}}{t\cdot j_m!}\right)\notag\\
&\quad=	\big[t^{d-2r-1}\big]\sum_{\substack{j_1,\ldots,j_m\in\N,j_{m+1}\in\N_0:\\j_1+\ldots+j_{m+1}= n}}\left(\frac{(t+1)(t+3)\cdot\ldots\cdot(t+2j_{m+1}-1)}{2^{j_{m+1}}j_{m+1}!}\cdot\frac{t^{\overline{j_1}}}{j_1!}\cdot\ldots\cdot\frac{t^{\overline{j_m}}}{ j_m!}\right),
\end{align*}
which follows directly from the definition of the coefficients $P_{j_1,\ldots,j_m}^{(n)}(r)$ in~\eqref{eq:coefficients_joint_convex_hull}. However, this sum was already evaluated in~\cite{GK20_Schlaefli_orthoschemes} in the proof of Proposition~3.8, or more precisely in (4.5) (with $k$ replaced by $d-2r-1$, $b$ replaced by $1$, and $j$ replaced by $m$) combined with (2.23) in~\cite{GK20_Schlaefli_orthoschemes}. This completes the proof.
\end{proof}

Theorem~\ref{theorem:exp_size_functionals_A} is proven similarly to Theorem~\ref{theorem:exp_size_functionals_B} using the corresponding results for random bridges instead of random walks. We only sketch the proof and point out the essential differences. 

\begin{proof}[Proof of Theorem~\ref{theorem:exp_size_functionals_A}]
Let $0\le l<m\le d-1$ be given. Using Remark~\ref{remark:faces_simplicial} and Lemma~\ref{lemma:independence_bridge}, we obtain
\begin{multline*}
\E Y_{m,l}(C_n^A)
	=\int_{G(d,d-l)}\sum_{1\le i_1<\ldots<i_m\le n-1}\frac{1}{2}\bP[\pos\{S_{i_1},\ldots,S_{i_m}\}\cap V\neq \{0\}]\\ \times\bP[\pos\{S_{i_1},\ldots,S_{i_m}\}\in\cF_m(C_n^A)]\,\nu_{d-l}(\dint V),\notag
\end{multline*}
in the same way as in the proof of Theorem~\ref{theorem:exp_size_functionals_B}. From now on, we ignore the integral since the sum under the integral will turn out to be constant for $\nu_{d-l}$-a.e.~$V\in G(d,d-l)$. Introducing the summation indices $j_1:= i_1$, $j_2:= i_2-i_1,\ldots,j_{m}:= i_{m}-i_{m-1}$, $j_{m+1} = n-i_m$, and recalling that $S_n=0$ a.s.\ we can rewrite the above sum as
\begin{align}\label{eq:sum_65}
&\sum_{\substack{j_1,\ldots,j_{m+1}\in\N:\\j_1+\ldots+j_{m+1} = n}}\frac 12\bP[\pos\{S_{j_1},\ldots,S_{j_1+\ldots+j_{m+1}}\}\cap V\neq\{0\}]\bP[\pos\{S_{j_1},\ldots,S_{j_1+\ldots+j_m}\}\in\cF_m(C_n^A)]\notag\\
&	\quad=\frac{1}{(m+1)!}\sum_{\sigma\in\text{Sym}(m+1)}\sum_{\substack{j_1,\ldots,j_{m+1}\in\N:\\j_1+\ldots+j_{m+1} = n}}\frac 12\bP[\pos\{S_{j_{\sigma(1)}},\ldots,S_{j_{\sigma(1)}+\ldots+j_{\sigma(m+1)}}\}\cap V\neq\{0\}]\\
&	\quad\quad\quad\times\bP[\pos\{S_{j_{\sigma(1)}},\ldots,S_{j_{\sigma(1)}+\ldots+j_{\sigma(m)}}\}\in\cF_m(C_n^A)],\notag
\end{align}
where we used that for each tuple $(j_1,\ldots,j_{m+1})\in\N^{m+1}$ with $j_1+\ldots+j_{m+1} =  n$ the tuple $(j_{\sigma(1)},\ldots,j_{\sigma(m+1)})$ satisfies the same condition. In the same way as in the proof of Theorem~\ref{theorem:exp_size_functionals_B}, we can use Corollary~\ref{cor:face_probabilities_pos_bridge} to observe that the probability $\bP[\pos\{S_{j_{\sigma(1)}},\ldots,S_{j_{\sigma(1)}+\ldots+j_{\sigma(m)}}\}\in\cF_m(C_n^A)]$ does not depend on the permutation $\sigma$. Introduce the random variables
\begin{align*}
T_1:=S_{j_1},T_2:=S_{j_1+j_2}-S_{j_1},\ldots,T_m:=S_{j_1+\ldots+j_m}-S_{j_1+\ldots+j_{m-1}},T_{m+1}=S_n-S_{j_1+\ldots+j_m},
\end{align*}
and a random and uniformly distributed permutation $\varsigma$ of $\{1,\ldots,m+1\}$ which is independent of $X_1,\ldots,X_n$. Using the exchangeability of $X_1,\ldots,X_n$, we observe that
\begin{align*}
\lefteqn{\frac{1}{(m+1)!}\sum_{\sigma\in\text{Sym}(m+1)}
\bP[\pos\{S_{j_{\sigma(1)}},\ldots,S_{j_{\sigma(1)}+\ldots+j_{\sigma(m+1)}}\}\cap V\neq\{0\}]}\\
&=\bP\big[\pos\{T_{\varsigma(1)}, T_{\varsigma(1)}+T_{\varsigma(2)},\ldots,T_{\varsigma(1)}+\ldots+T_{\varsigma(m+1)}\}\cap V\neq\{0\}\big]\\
&=\bP\big[\pos\{T_{\varsigma(1)}, T_{\varsigma(1)}+T_{\varsigma(2)},\ldots,T_{\varsigma(1)}+\ldots+T_{\varsigma(m)}\}\cap V\neq\{0\}\big]\\
&=\frac{2}{(m+1)!}\sum_{r=1}^\infty\stirling {m+1}{l+2r}.
\end{align*}
Note that we applied Remark~\ref{remark:prob_walkbridge_intersects_V} (with $n$ replaced by $m+1$) in the last equation. Note also that Remark~\ref{remark:prob_walkbridge_intersects_V}  was applicable (for $\nu_{d-l}$-a.e~$V\in G(d,d-l)$) since $T_{\varsigma(1)},\ldots,T_{\varsigma(m+1)}$ satisfy the conditions (Ex), (Br) and (GP') (again with $n$ replaced by $m+1$).  Combining all of these results and inserting them into~\eqref{eq:sum_65} yields
\begin{align*}
\E Y_{m,l}(C_n^A)
&	=\frac 12\left(\frac{2}{(m+1)!}\sum_{r=1}^\infty\stirling {m+1}{l+2r}\right)\sum_{\substack{j_1,\ldots,j_{m+1}\in\N:\\j_1+\ldots+j_{m+1}= n}}\bP[\pos\{S_{j_1},\ldots,S_{j_1+\ldots+j_{m}}\}\in\cF_m(C_n^A)]\\
&	=\left(\frac{2}{(m+1)!}\sum_{r=1}^\infty\stirling {m+1}{l+2r}\right)\sum_{r=0}^\infty\sum_{\substack{j_1,\ldots,j_{m+1}\in\N:\\j_1+\ldots+j_{m+1}=n}}\frac{ Q_{j_1,\ldots,j_m}^{(n)}(d-m-2r-1)}{j_1!\cdot\ldots\cdot j_{m+1}! },
\end{align*}
where we used Corollary~\ref{cor:face_probabilities_pos_bridge} in the last equation, while the coefficients  $Q_{j_1,\ldots,j_m}^{(n)}$ are defined in~\eqref{eq:coefficients_joint_convex_hull_bridge}.
In order to prove~\eqref{eq:size_fct_A}, it remains to show that
\begin{align}\label{eq:to_prove_A}
\sum_{\substack{j_1,\ldots,j_{m+1}\in\N:\\j_1+\ldots+j_{m+1}=n}}\frac{ Q_{j_1,\ldots,j_m}^{(n)}(d-m-2r-1)}{j_1!\cdot\ldots\cdot j_{m+1}! }
	=\frac{(m+1)!}{n!} \stirling n{d-2r}\stirlingsec{d-2r}{m+1}
\end{align}
holds for all admissible $r$. Again, let   $t^{\overline{j}}=t(t+1)\cdot\ldots\cdot (t+j-1)$ denote the rising factorial. Then, using the definition of $Q_{j_1,\ldots,j_m}^{(n)}$ we have
\begin{align*}
\sum_{\substack{j_1,\ldots,j_{m+1}\in\N:\\j_1+\ldots+j_{m+1}=n}}\frac{ Q_{j_1,\ldots,j_m}^{(n)}(d-m-2r-1)}{j_1!\cdot\ldots\cdot j_{m+1}! }
&	=\big[t^{d-m-2r-1}\big]\sum_{\substack{j_1,\ldots,j_{m+1}\in\N:\\j_1+\ldots+j_{m+1}=n}}\frac{t^{\overline{j_1}}}{t\cdot j_1!}\ldots \frac{t^{\overline{j_{m+1}}}}{t\cdot j_{m+1}!}\notag\\
&	=\big[t^{d-2r}\big]\sum_{\substack{j_1,\ldots,j_{m+1}\in\N:\\j_1+\ldots+j_{m+1}=n}}\frac{t^{\overline{j_1}}}{ j_1!}\ldots \frac{t^{\overline{j_{m+1}}}}{j_{m+1}!}.
\end{align*}
Again, this sum was already evaluated in the proof of Propositions 3.8 in~\cite{GK20_Schlaefli_orthoschemes}. In particular,~\eqref{eq:to_prove_A} follows from~\cite[Eq.~(4.4)]{GK20_Schlaefli_orthoschemes} (with $k$ replaced by $d-2r$, $b$ replaced by $0$, and $j$ replaced by $m+1$).
This completes the proof.
\end{proof}

\subsection{Expected size functionals \texorpdfstring{$Z_{j,k}$}{Z\_\{j,k\}}: Proof of Theorems~\ref{theorem:gen_angle_sums_quermass_A} and~\ref{theorem:gen_angle_sums_quermass_B}}\label{sec:proof_gen_angle_sums}

In this section, we prove Theorems~\ref{theorem:gen_angle_sums_quermass_A} and~\ref{theorem:gen_angle_sums_quermass_B} on the conic quermassintegral sums over the tangent cones of $C_n^A$ and $C_n^B$. We need to state several lemmas first.

\begin{lemma}\label{lem:side_characterization_proj}
Let $s_1,\ldots,s_n\in \R^d$, $n\geq d$, be vectors in general position (meaning that any $d$ of them are linearly independent) and let $C:= \pos\{s_1,\ldots,s_n\}$ be their positive hull.   Fix some $k\in\{1,\ldots,d-1\}$, let $1\le i_1<\ldots<i_k\le n$ be any indices and define $M:=\lin^\perp\{s_{i_1},\ldots,s_{i_k}\}$.
Then,
\begin{align}\label{eq:equivalence_proj_walk_special}
\pos\{s_{i_1},\ldots,s_{i_k}\}\in\cF_k(C)\;\;\Leftrightarrow\;\; C|M \neq M.
\end{align}
\end{lemma}
\begin{proof}
First suppose that $\pos\{s_{i_1},\ldots,s_{i_k}\}\in\cF_k(C)$.
Thus, there is a supporting hyperplane $H=u^\perp$ for some $u\in\R^d\backslash\{0\}$ such that $u^-:=\{y\in \R^d:\langle y, u\rangle \leq 0\} \supseteq C$ and $H\cap C=\pos\{s_{i_1},\ldots,s_{i_k}\}$.
This implies that $u\in M$, and thus, for all $x\in C$
\begin{align*}
\langle u,(x|M)\rangle=\langle u,(x|M) -x\rangle+\langle u,x\rangle =\langle u,x\rangle \le 0.
\end{align*}
This implies $C|M\neq M$.

To prove the converse implication suppose that $C|M\neq M$. Observe that $C|M=\pos \{s_i|M:i\notin\{i_1,\dots,i_k\}\}$ and each $d-k$ of the vectors $s_i|M$, $i\notin\{i_1,\dots,i_k\}$ are linearly independent. Hence,  $C|M\neq M$ implies that $C|M$ is pointed; see Remark~\ref{remark:pointed}. Thus, $\{0\}$ is a $0$-face of $C|M$. This yields a supporting hyperplane $H = u^\bot\cap M \subseteq M$ with $u\in M$ such that $H\cap (C|M) =\{0\}$ and $u^-\cap M \supseteq  C|M$. We claim that $u^\perp = H + M^\bot \in G(d,d-1)$ is a supporting hyperplane of $C$.
Indeed, we have
$$
u^- = (u^-\cap M) + M^\bot \supseteq (C|M)+M^\perp\supseteq C.
$$
It remains to show that $u^\bot \cap C = \pos\{s_{i_1},\dots,s_{i_k}\}$.
Clearly, $u^\bot\cap C\supseteq \pos\{s_{i_1},\dots,s_{i_k}\}$. Assume that this inclusion is strict.   Since $u^\bot$ is a supporting hyperplane of $C$, $u^\bot\cap C$ defines a face of $C$. Thus, by the same argument as in Remark~\ref{remark:faces_simplicial}, there is an index $j\in\{1,\dots,n\}\backslash\{i_1,\dots,i_k\}$ such that $s_j\in u^\bot\cap C$. But this implies $s_j|M\in (C|M)\cap H=\{0\}$ and therefore $s_j\in M^\perp$. This is a contradiction to the general position assumption since $s_{i_1},\dots,s_{i_k}$ also belong to the $k$-dimensional linear subspace $M^\bot$, and they are all linearly independent. This completes the proof.
\end{proof}

\begin{lemma}\label{lemma:projection_walk}
Let $S_1,\ldots,S_n$ be a random walk in $\R^d$ whose increments $X_1,\ldots,X_n$ satisfy assumptions $\textup{(GP)}$ and $\textup{($\pm$Ex)}$. Fix some $k\in\{1,\ldots,d-1\}$, let $1\le i_1<\ldots<i_k\le n$ be any indices and define $M:=\lin^\perp\{S_{i_1},\ldots,S_{i_k}\}$. Then, outside an event of probability $0$, it holds that
\begin{align}\label{eq:equivalence_proj_walk}
\pos\{S_{i_1},\ldots,S_{i_k}\}\in\cF_k(C_n^B)\;\;\Leftrightarrow\;\; C_n^B|M\neq M.
\end{align}
Also, projecting $S_1,\ldots,S_n$ on $M$ yields $k$ random bridges and one random walk in $M$ in the sense of Section~\ref{sec:prob_positive_hulls}. More precisely, the increments of the $k$ random bridges are given by
\begin{align}\label{eq:proj_bridges}
&Y_1^{(1)}=X_{1}|M, \;\; Y_2^{(1)}=X_{2}|M, \;\; \ldots,\;\; Y_{i_1}^{(1)}=X_{i_1}|M,\notag\\
&\hspace*{4cm}\vdots\\
&Y_1^{(k)}=X_{i_{k-1}+1}|M, \;\; Y_2^{(k)}=X_{i_{k-1}+2}|M, \;\; \ldots, \;\; Y_{i_k-i_{k-1}}^{(k)}=X_{i_{k}}|M,\notag
\end{align}
while the increments of the random walk are given by
\begin{align}\label{eq:proj_walk}
X_1^{(1)}=X_{i_{k}+1}|M,\;\; X_2^{(1)}=X_{i_{k}+2}|M, \;\; \ldots,\;\; X_{n-i_k}^{(1)}=X_{n}|M.
\end{align}
The random bridges and the random walk in $M$ corresponding to these increments satisfy the exchangeability
condition~\eqref{eq:exchangeability_joint_walks_bridges} and the general position assumption of Theorem~\ref{theorem:absorpt_prob_joint_convex hull}. The latter means  that each $d-k$ of the vectors $S_i|M$, $i\in\{1,\dots,n\}\backslash\{i_1,\dots,i_k\}$ are linearly independent a.s.
\end{lemma}
\begin{proof}
The equivalence~\eqref{eq:equivalence_proj_walk} is a special case of Lemma~\ref{lem:side_characterization_proj}.
The other assertions were shown in the proof of Theorem~1.6 of~\cite{KVZ17}.
\end{proof}

A similar claim in the $A$-case follows from Theorem~1.11 from~\cite{KVZ17} and reads as follows.

\begin{lemma}\label{lemma:projection_bridge}
Let $S_1,\ldots,S_n$ be a random bridge in $\R^d$ whose increments $X_1,\ldots,X_n$ satisfy assumptions $\textup{(Br)}$, $\textup{(GP')}$ and $\textup{(Ex)}$. Fix some $k\in\{1,\ldots,d-1\}$, let $1\le i_1<\ldots<i_k\le n-1$ be any indices and define $M:=\lin^\perp\{S_{i_1},\ldots,S_{i_k}\}$. Then, outside an event of probability $0$ it holds that
\begin{align}\label{eq:equivalence_proj_bridge}
\pos\{S_{i_1},\ldots,S_{i_k}\}\in\cF_k(C_n^A)\;\;\Leftrightarrow\;\; C_n^A|M\neq M.
\end{align}
Also, projecting $S_1,\ldots,S_n$ on $M$ yields $k+1$ random bridges in $M$ in the sense of
Section~\ref{sec:prob_positive_hulls}. More precisely, the increments of these random bridges are given by
\begin{align*}
&Y_1^{(1)}=X_{1}|M,\;\; Y_2^{(1)}=X_{2}|M, \;\; \ldots,\;\; Y_{i_1}^{(1)}=X_{i_1}|M,\notag\\
&\hspace*{4cm}\vdots\\
&Y_1^{(k)}=X_{i_{k-1}+1}|M, \;\; Y_2^{(k)}=X_{i_{k-1}+2}|M, \;\; \ldots,\;\; Y_{i_k-i_{k-1}}^{(k)}=X_{i_{k}}|M,\notag,\\
&Y_1^{(k+1)}=X_{i_{k}+1}|M, \;\; Y_2^{(k+1)}=X_{i_{k}+2}|M, \;\; \ldots,\;\; Y_{n-i_k}^{(k+1)}=X_{n}|M.
\end{align*}
The random bridges in $M$ corresponding to these increments satisfy the exchangeability
condition~\eqref{eq:exchangeability_joint_walks_bridges} and the general position assumption of Theorem~\ref{theorem:absorpt_prob_joint_convex hull}. The latter means  that each $d-k$ of the vectors $S_i|M$, $i\in\{1,\dots,n-1\}\backslash\{i_1,\dots,i_k\}$ are linearly independent a.s.
\end{lemma}


\begin{proof}[Proof of Theorem~\ref{theorem:gen_angle_sums_quermass_B}]
We divide this proof into three separate cases, the first case being $j=0$ and $k\in\{0,\ldots,d\}$, the second case $1\le j\le k\le d-1$, and the last case $j\in\{1,\ldots,d\}$ and $k=d$.

\vspace*{2mm}
\noindent
\textsc{Case 1.}
For $j=0$ and $k\in\{0,\ldots,d\}$, we are interested in the quantity
\begin{align*}
\E\sum_{F\in\cF_0(C_n^B)}U_k\big(T_F(C_n^B)\big)=\E\big[ U_k(C_n^B)\1_{\{C_n^B\neq\R^d\}}\big].
\end{align*}
By Remark~\ref{remark:pointed}, $C_n^B$ is pointed (and, in particular, not a linear subspace)
on the event that $C_n^B\neq\R^d$. For $k=d$, we have
\begin{align*}
\E\big[U_d(C_n^B)\1_{\{C_n^B\neq\R^d\}}\big]
=\frac 12 \E\big[\1_{\{C_n^B\cap\{0\}\neq\{0\}\}}\1_{\{C_n^B\neq\R^d\}}\big]=0.
\end{align*}
This proves Theorem~\ref{theorem:gen_angle_sums_quermass_B} in the case $j=0$, $k=d$. For $j=0$ and $k\in\{0,\ldots,d-1\}$, we obtain that
\begin{align}\label{eq:exp_ind_1}
\E[U_k(C_n^B)\1_{\{C_n^B\neq\R^d\}}]
&	=\frac 12\E\left[\int_{G(d,d-k)}\1_{\{C_n^B\cap V\neq\{0\}\}}\,\nu_{d-k}(\dint V)\1_{\{C_n^B\neq\R^d\}}\right]\notag\\
&	=\frac 12\int_{G(d,d-k)}\bP\big[C_n^B\cap V\neq\{0\},C_n^B\neq\R^d\big]\,\nu_{d-k}(\dint V)\notag\\
&	=\frac 12\int_{G(d,d-k)}\bP\big[C_n^B\cap V\neq\{0\}\big]\nu_{d-k}(\dint V)-\frac 12\bP[C_n^B=\R^d],
\end{align}
where we used that
$
\{C_n^B=\R^d\big\}\subseteq \{C_n^B\cap V\neq \{0\}\}.
$
Next, we can apply Remark~\ref{remark:prob_walkbridge_intersects_V}, and Equation~\eqref{eq:abs_prob_walk} to arrive at
\begin{align}\label{eq:U_l_with_ind}
\E[U_k(C_n^B)\1_{\{C_n^B\neq\R^d\}}]
&	=\frac{1}{2^nn!}\left(\sum_{r=0}^\infty\stirlingb{n}{k+2r+1}-\sum_{r=0}^\infty\stirlingb{n}{d+2r+1}\right)\notag\\
&	=\frac{1}{2^nn!}\left(\sum_{r=0}^\infty\stirlingb{n}{d-2r-1}-\sum_{r=0}^\infty\stirlingb{n}{k-2r-1}\right),
\end{align}
where in the second equality we used that
$$
\stirlingb n1+\stirlingb n3+\ldots=\stirlingb n0+\stirlingb n2+\ldots=2^{n-1}n!.
$$
Note that Remark~\ref{remark:prob_walkbridge_intersects_V} was applicable since $\nu_{d-k}$-a.e.~$V\in G(d,d-k)$ is in general position with respect to any fixed linear subspace; see~\cite[Lemma~13.2.1]{SW08}.

\vspace*{2mm}
\noindent
\textsc{Case 2.}
Now, let $1\le j\le k\le d-1$ and let $F\in\cF_j(C_n^B)$ be a $j$-face of $C_n^B$. Then, $F=\pos\{S_{i_1},\ldots,S_{i_j}\}$ for some $1\le i_1<\ldots<i_j\le n$ due to Remark~\ref{remark:faces_simplicial}. Let $M:=\lin^\perp\{S_{i_1},\ldots,S_{i_j}\}$. It follows from~\eqref{eq:tangent_cone_decomposition} that the tangent cone of $C_n^B$ at $F$ is given by
\begin{align*}
T_F\big(C_n^B\big)=M^\perp + \big(C_n^B|M\big).
\end{align*}
Using this relation, we obtain
\begin{align*}
\E\sum_{F\in\cF_j(C_n^B)}U_k\big(T_F(C_n^B)\big)
	&=\E\sum_{1\le i_1<\ldots<i_j\le n}U_k\big(T_{\pos\{S_{i_1},\ldots,S_{i_j}\}}\big(C_n^B\big)\big)
		\1_{\{\pos\{S_{i_1},\ldots,S_{i_j}\}\in\cF_j(C_n^B)\}}\\
	&=\E\sum_{1\le i_1<\ldots<i_j\le n}U_k\big(M^\perp+(C_n^B|M))\1_{\{\pos\{S_{i_1},\ldots,S_{i_j}\}\in\cF_j(C_n^B)\}}.
\end{align*}
For every cone $C\subset M$, the conic Crofton formula~\eqref{eq:crofton} and the properties of the conic intrinsic volumes~\cite[(2.9)]{amelunxen_edge} yield
$$
U_k(M^\perp+C)
=
\sum_{r=0}^\infty \upsilon_{k + 2r + 1} (M^\perp+C)
=
\sum_{r=0}^\infty \upsilon_{k-j + 2r + 1} (C)
=
U_{k-j} (C).
$$
In particular, $U_k(M^\perp+(C_n^B|M))=U_{k-j}(C_n^B|M)$ holds true. Using~\eqref{eq:equivalence_proj_walk}  yields
\begin{align}
&	\E\sum_{F\in\cF_j(C_n^B)}U_k\big(T_F(C_n^B)\big)\notag\\
&	\quad=\E\sum_{1\le i_1<\ldots<i_j\le n}U_{k-j}\big(C_n^B|M\big)\1_{\{C_n^B|M\neq M\}}\notag\\
&	\quad=\E\sum_{1\le i_1<\ldots<i_j\le n}\frac 12\int_{G(d,d-k+j)}\1_{\{(C_n^B|M)\cap V\neq\{0\}\}}\,\nu_{d-k+j}(\dint V) \1_{\{C_n^B|M\neq M\}}\notag\\
&	\quad=\frac 12\int_{G(d,d-k+j)}\sum_{1\le i_1<\ldots<i_j\le n}\P\big[(C_n^B|M)\cap V\neq\{0\}\big]\,\nu_{d-k+j}(\dint V)\label{eq:sums_gamma}\\
&	\quad\quad-\frac 12\int_{G(d,d-k+j)}\sum_{1\le i_1<\ldots<i_j\le n}\P\big[(C_n^B|M)\cap V\neq\{0\}, (C_n^B|M)= M\big]\,\nu_{d-k+j}(\dint V).\label{eq:sums_gamma1}
\end{align}
Lemma~\ref{lemma:projection_walk} yields that $C_n^B|M$ is the joint positive hull of $j$ random bridges (of lengths $i_1,i_2-i_1,\ldots,i_j-i_{j-1}$) and $1$ random walk (of length $n-i_j$) in $M$. Thus, using Corollary~\ref{cor:prob_walkbridge_intersects_V}, the term in~\eqref{eq:sums_gamma} can be rewritten as follows:
\begin{align*}
\sum_{1\le i_1<\ldots<i_j\le n}\sum_{r=0}^\infty \frac{P_{i_1,i_2-i_1,\ldots,i_j-i_{j-1}}^{(n)}(k-j+2r+1)}{i_1!(i_2-i_1)!\cdot\ldots\cdot (i_j-i_{j-1})!(n-i_j)!2^{n-i_j}},
\end{align*}
where we used that $\nu_{d-k+j}$-a.e. linear subspace $V$ satisfies the general position assumption from Corollary~\ref{cor:prob_walkbridge_intersects_V}.
Here, the coefficients $P_{l_1,\ldots,l_j} ^{(n)}(r)$ are as defined in Corollary~\ref{cor:face_probabilities_pos_walk}.
We can simplify the above sums to get
\begin{align*}
\sum_{r=0}^\infty\:\sum_{\substack{l_1,\ldots,l_j\in\N,l_{j+1}\in\N_0\\l_1+\ldots+l_{j+1}=n}}\frac{P^{(n)}_{l_1,\ldots,l_j}(k-j+2r+1)}{l_1!\cdot\ldots\cdot l_{j+1}!2^{l_{j+1}}}=\frac{j!}{2^{n-j}n!}\sum_{r=0}^\infty\stirlingb n{k+2r+1}\stirlingsecb{k+2r+1}j,
\end{align*}
where the last equation follows from~\eqref{eq:to_prove2_size_funct}.

Since $k<d$, we have $\dim (M\cap V)=d-k>0$ for $\nu_{d-k+j}$-a.e~$V\in G(d,d-k+j)$. Thus, together with the equivalence~\eqref{eq:equivalence_proj_walk}, the term in~\eqref{eq:sums_gamma1} can be rewritten as
\begin{align*}
\frac 12\sum_{1\le i_1<\ldots<i_j\le n} \P\big[(C_n^B|M) = M\big]
& =\frac 12\sum_{1\le i_1<\ldots<i_j\le n} \P[\pos\{S_{i_1},\ldots,S_{i_j}\}\notin\cF_j(C_n^B)]\\
&	=	\sum_{\substack{l_1,\ldots,l_j\in\N,l_{j+1}\in\N_0:\\l_1+\ldots+l_{j+1}=n}}\sum_{r=0}^\infty\frac{P_{l_1,\ldots,l_j} ^{(n)}(d-j+2r+1)}{l_1!\cdot\ldots\cdot l_{j+1}! 2^{l_{j+1}}}		\\
&	=\frac{j!}{2^{n-j}n!}\sum_{r=0}^\infty \stirlingb n{d+2r+1}\stirlingsecb{d+2r+1}j.
\end{align*}
Note that the second equality follows from Corollary~\ref{cor:face_probabilities_pos_walk} and the last equality follows again from~\eqref{eq:to_prove2_size_funct}. Taking all into consideration, this yields
\begin{align*}
&\E\sum_{F\in\cF_j(C_n^B)}U_k\big(T_F(C_n^B)\big)\\
&\quad=\frac{j!}{2^{n-j} n!}\left(\sum_{r=0}^\infty\stirlingb n{k+2r+1}\stirlingsecb {k+2r+1}j
-
\sum_{r=0}^\infty\stirlingb n{d+2r+1}\stirlingsecb {d+2r+1}j\right).
\end{align*}
To complete the proof of Theorem~\ref{theorem:gen_angle_sums_quermass_B} for $1\le j\le k\le d-1$, we rewrite this  using the fact that
\begin{align*}
\sum_{k=0}^n(-1)^k\stirlingb nk\stirlingsecb kj=0, \quad \sum_{k=0}^n\stirlingb nk\stirlingsecb kj=\frac{2^nn!}{2^jj!}\binom nj,
\end{align*}
which was proven, for example, in~\cite[Corollary~3.13, Eq.~(2.23)]{GK20_Schlaefli_orthoschemes}.

\vspace*{2mm}
\noindent
\textsc{Case 3.}
For $k=d$ and $j\in\{1,\ldots,d-1\}$ the claim follows from the fact that $U_d(C) = 0$ for every cone $C\subseteq \R^d$.
\end{proof}

The proof of Theorem~\ref{theorem:gen_angle_sums_quermass_A} is similar to that of Theorem~\ref{theorem:gen_angle_sums_quermass_B} and we will only sketch it.

\begin{proof}[Proof of Theorem~\ref{theorem:gen_angle_sums_quermass_A}]
We divide this proof into the same three different cases as in the $B$-case.

\vspace*{2mm}
\noindent
\textsc{Case 1.}
For $j=0$ and $k=d$, the claim is easily checked. For $j=0$ and $k\in\{0,\ldots,d-1\}$, we follow the arguments of the $B$-case to obtain
\begin{align*}
\E\sum_{F\in\cF_0(C_n^A)}U_k\big(T_F(C_n^A)\big)
&	=\E\big[ U_k(C_n^A)\1_{\{C_n^A\neq\R^d\}}\big]\\
&	=\frac 12\int_{G(d,d-k)}\bP\big[C_n^A\cap V\neq\{0\}\big]\nu_{d-k}(\dint V)-\frac 12\bP[C_n^A=\R^d]
\end{align*}
Applying~\eqref{eq:abs_prob_bridge} and Remark~\ref{remark:prob_walkbridge_intersects_V}, for $\nu_{d-k}$-a.e. $V$ yields
\begin{align*}
\E\sum_{F\in\cF_0(C_n^A)}U_k\big(T_F(C_n^A)\big)=\frac{1}{n!}\left(\sum_{r=1}^\infty\stirling n{l+2r}-\sum_{r=1}^\infty \stirling n{d+2r}\right),
\end{align*}
which proves the claim for $C_n^A$.

\vspace*{2mm}
\noindent
\textsc{Case 2.}
Now, let $1\le j\le k\le d-1$ and let $F\in\cF_j(C_n^A)$ be a $j$-face of $C_n^A$. Then, $F=\pos\{S_{i_1},\ldots,S_{i_j}\}$ for some $1\le i_1<\ldots<i_j\le n-1$. Also, for $M:=\lin^\perp\{S_{i_1},\ldots,S_{i_j}\}$, the tangent cone of $C_n^A$ at $F$ is again given as the orthogonal sum
\begin{align*}
T_F\big(C_n^A\big)=M^\perp + \big(C_n^A|M\big).
\end{align*}
This yields
\begin{align*}
\E\sum_{F\in\cF_j(C_n^A)}U_k\big(T_F(C_n^A)\big)
	&=\E\sum_{1\le i_1<\ldots<i_j\le n-1}U_k\big(M^\perp+(C_n^A|M)\big)\1_{\{\pos\{S_{i_1},\ldots,S_{i_j}\}\in\cF_j(C_n^A)\}}.
\end{align*}
On the event that $\pos\{S_{i_1},\ldots,S_{i_j}\}\in\cF_j(C_n^A)$, which is equal to the event $\{C_n^A|M\neq M\}$ by Lemma~\ref{lemma:projection_bridge}, the tangent cone $T_{\{\pos\{S_{i_1},\ldots,S_{i_j}\}}(C_n^A)$ is not a linear subspace. Thus, we can use the reasoning of the $B$-case to obtain
\begin{align}
&\E\sum_{F\in\cF_j(C_n^A)}U_k\big(T_F(C_n^A)\big)\notag\\
&	\quad=\E\sum_{1\le i_1<\ldots<i_j\le n-1}U_{k-j}\big(C_n^A|M\big)\1_{\{C_n^A|M\neq M\}}\notag\\
&	\quad=\int_{G(d,d-k+j)}\sum_{1\le i_1<\ldots<i_j\le n-1}\frac 12\P\big[(C_n^A|M)\cap V\neq\{0\}\big]\,\nu_{d-k+j}(\dint V)\label{eq:sums_gamma_A}\\
&	\quad\quad-\int_{G(d,d-k+j)}\sum_{1\le i_1<\ldots<i_j\le n-1}\frac 12\P\big[(C_n^A|M)\cap V\neq\{0\}, C_n^A|M= M\big]\,\nu_{d-k+j}(\dint V).\label{eq:sums_gamma1_A}
\end{align}
Due to Lemma~\ref{lemma:projection_bridge}, the projection $C_n^A|M$ is the joint positive hull of $j+1$ random bridges (of lengths $i_1,i_2-i_1,\ldots,i_j-i_{j-1}, n-i_j$) in $M$. Thus, using Corollary~\ref{cor:prob_walkbridge_intersects_V}, the integral in line~\eqref{eq:sums_gamma_A} can be rewritten as
\begin{align*}
\sum_{1\le i_1<\ldots<i_j\le n-1}\:\sum_{r=0}^\infty \frac{Q_{i_1,i_2-i_1,\ldots,i_j-i_{j-1}}^{(n)}(k-j+2r+1)}{i_1!(i_2-i_1)!\cdot\ldots\cdot (i_j-i_{j-1})!(n-i_j)!},
\end{align*}
for $\nu_{d-k+j}$-a.e. $V$. Here, the coefficients $Q_{j_1,\ldots,j_k} ^{(n)}(r)$ are as defined in Corollary~\ref{cor:face_probabilities_pos_bridge}.
We can simplify the above sums to get
\begin{align*}
\sum_{r=0}^\infty\:\sum_{\substack{l_1,\ldots,l_{j+1}\in\N\\l_1+\ldots+l_{j+1}=n}}\frac{Q^{(n)}_{l_1,\ldots,l_j}(k-j+2r+1)}{l_1!\cdot\ldots\cdot l_{j+1}!}=\frac{(j+1)!}{n!}\sum_{r=1}^\infty\stirling n{k+2r}\stirlingsec{k+2r}{j+1}
\end{align*}
where the last step follows from~\eqref{eq:to_prove_A}. In the same way as in the $B$-case, we can argue that the integral in line~\eqref{eq:sums_gamma1_A} simplifies to
\begin{align*}
\frac 12\sum_{1\le i_1<\ldots<i_j\le n-1} \P[\pos\{S_{i_1},\ldots,S_{i_j}\}\notin\cF_j(C_n^A)]
&	=	\sum_{\substack{l_1,\ldots,l_{j+1}\in\N:\\l_1+\ldots+l_{j+1}=n}}\:\sum_{r=0}^\infty\frac{Q_{l_1,\ldots,l_j} ^{(n)}(d-j+2r+1)}{l_1!\cdot\ldots\cdot l_{j+1}! }		\\
&	=\frac{(j+1)!}{n!}\sum_{r=1}^\infty \stirling n{d+2r}\stirlingsec{d+2r}{j+1},
\end{align*}
where the first equality follows from Corollary~\ref{cor:face_probabilities_pos_bridge} and the second equality again from~\eqref{eq:to_prove_A}.
Taking all into consideration, this yields
\begin{align*}
\E\sum_{F\in\cF_j(C_n^A)}U_k\big(T_F(C_n^A)\big)
&	=\frac{(j+1)!}{n!}\bigg(\sum_{r=1}^\infty\stirling n{k+2r}\stirlingsec{k+2r}{j+1}-\sum_{r=1}^\infty \stirling n{d+2r}\stirlingsec{d+2r}{j+1}\bigg)\\
&	=\frac{(j+1)!}{n!}\bigg(\sum_{r=0}^\infty \stirling n{d-2r}\stirlingsec{d-2r}{j+1}-\sum_{r=0}^\infty\stirling n{k-2r}\stirlingsec{k-2r}{j+1}\bigg),
\end{align*}
where the last equation follows from the identities
\begin{align*}
\sum_{k=0}^n(-1)^k\stirling nk\stirlingsec k{j+1}=0, \quad \sum_{k=0}^n\stirling nk\stirlingsec k{j+1}=\frac{n!}{(j+1)!}\binom{n-1}{j},
\end{align*}
which follow from~\cite[Corollary~3.13, Eq.~(2.23)]{GK20_Schlaefli_orthoschemes}. This proves the claim for $C_n^A$.

\vspace*{2mm}
\noindent
\textsc{Case 3.}
For $k=d$ and $j\in\{1,\ldots,d\}$, we obtain
\begin{align*}
\E\sum_{F\in\cF_j(C_n^B)}U_d\big(T_F(C_n^B)\big)
	=0,\quad \E\sum_{F\in\cF_j(\widetilde C_n^B)}U_d\big(T_F(\widetilde C_n^B)\big)
	=0
\end{align*}
in the same way as in the $B$-case. This completes the proof.
\end{proof}

\subsection{Conic intrinsic volume sums: Proofs of Corollaries~\ref{cor:Y_intvol_A}, \ref{corollary:gen_angle_sums_int_vol_A}, \ref{cor:Y_intvol_B} and~\ref{corollary:gen_angle_sums_int_vol_B}}
\label{sec:conic_intrinsic_vol_sums_proofs}
Corollaries~\ref{corollary:gen_angle_sums_int_vol_A} and~\ref{corollary:gen_angle_sums_int_vol_B} on the generalized intrinsic volume sums follow from Theorems~\ref{theorem:gen_angle_sums_quermass_A} and~\ref{theorem:gen_angle_sums_quermass_B}, respectively. For the proof, we use the relations
\begin{align}\label{eq:conic_Crofton}
\upsilon_d(C)=U_{d-1}(C),\quad \upsilon_{d-1}(C)=U_{d-2}(C),\quad\upsilon_k(C)=U_{k-1}(C)-U_{k+1}(C),
\end{align}
for $k\in\{1,\ldots,d-2\}$,
which hold for every cone $C\subseteq \R^d$ and follow from the conic Crofton formula~\eqref{eq:crofton}. We only prove the $B$-case (Corollary~\ref{corollary:gen_angle_sums_int_vol_B}) since the $A$-case is analogous.

\begin{proof}[Proof of Corollary~\ref{corollary:gen_angle_sums_int_vol_B}]
The case $k\in\{j+1,\ldots,d-1\}$ directly follows from~\eqref{eq:conic_Crofton} together with Theorem~\ref{theorem:gen_angle_sums_quermass_B}. For $k=d$, we obtain
\begin{align*}
&\E\sum_{F\in\cF_j(C_n^B)}\upsilon_d\big(T_F(C_n^B)\big)\\
&	\quad=\E\sum_{F\in\cF_j(C_n^B)}U_{d-1}\big(T_F(C_n^B)\big)\\
&	\quad=\frac{j!}{2^{n-j}n!}\bigg(\sum_{r=0}^\infty\stirlingb n{d-2r-1}\stirlingsecb{d-2r-1}{j}-\sum_{r=0}^\infty\stirlingb n{d-2r-2}\stirlingsecb{d-2r-2}{j}\bigg)\\
&	\quad=\frac{j!}{2^{n-j}n!}\sum_{r=0}^\infty (-1)^{r}\stirlingb n{d-1-r}\stirlingsecb{d-1-r}{j}.
\end{align*}
For $k=j\in\{0,\ldots,d-1\}$, we use that $T_F(C_n^B)$ has a $j$-dimensional lineality space, defined as the linear subspace of maximal dimension contained in $T_F(C_n^B)$. Thus, we have $\upsilon_j(T_F(C_n^B))+\ldots+\upsilon_d(T_F(C_n^B))=1$, which yields
\begin{align*}
&\E\sum_{F\in\cF_j(C_n^B)}\upsilon_j\big(T_F(C_n^B)\big)\\
&	\quad=\E f_j(C_n^B)-\sum_{k=j+1}^d\E\sum_{F\in\cF_j(C_n^B)}\upsilon_k\big(T_F(C_n^B)\big)\\
&	\quad=\frac{2\cdot j!}{2^{n-j}n!}\sum_{r=0}^\infty\stirlingb n{d-2r-1}\stirlingsecb {d-2r-1}{j}-\frac{j!}{2^{n-j}n!}\sum_{k=j+1}^{d-1}\stirlingb n{k}\stirlingsecb{k}{j}\\
&	\quad\quad-\frac{j!}{2^{n-j}n!}\sum_{r=0}^\infty (-1)^{r}\stirlingb n{d-1-r}\stirlingsecb{d-1-r}{j}\\
&	\quad=\frac{j!}{2^{n-j}n!}\stirlingb nj \stirlingsecb jj.
\end{align*}
Note that the second equation follows from the previous cases and Corollary~\ref{cor:exp_faces_A}.
\end{proof}

Corollaries~\ref{cor:Y_intvol_A} and~\ref{cor:Y_intvol_B} follow from Theorems~\ref{theorem:exp_size_functionals_A} and~\ref{theorem:exp_size_functionals_B}, respectively, using the relation~\eqref{eq:conic_Crofton} as above.

\subsection{Sufficient condition for general position}\label{sec:proof_sufficient_condition}
We now state and prove a natural condition on the distribution of $(X_1,\ldots,X_n)$ under which the general position assumptions
$\textup{(GP)}$ and, in some sense, also $\textup{(GP')}$ hold. In~\cite[Lemma~5.1]{GK2020_WeylCones}, we proved that the same condition implies the general position assumptions $\textup{(GP*)}$ and $\textup{(GP**)}$.

\begin{lemma}\label{lemma:sufficent_condition}
Let $\mu$ be a $\sigma$-finite Borel measure on $\R^d$ that assigns measure zero to each affine hyperplane, i.e.~ each $(d-1)$-dimensional affine subspace. Furthermore, let $X_1,\ldots,X_n$ be random vectors in $\R^d$ having a joint $\mu^n$-density on $(\R^d)^n$. Then $X_1,\ldots,X_n$ satisfy the general position assumption $\textup{(GP)}$ provided $n\geq d$.

Furthermore, let $\widetilde X_i:=X_i-\frac 1n (X_1+\ldots+X_n)$, for $i=1,\ldots,n$ and let $n\geq d+1$.  Then, the partial sums
$$
\widetilde S_i := \widetilde X_1 + \ldots +  \widetilde X_i = S_i-\frac i	n S_n,\quad i=1,\ldots,n,
$$
satisfy the general position assumption $\textup{(GP')}$, and also, the bridge property $\textup{(Br)}$.
\end{lemma}

Note that if we additionally assume $X_1,\ldots,X_n$ to be exchangeable, then also the increments $\widetilde X_1,\ldots,\widetilde X_n$ are exchangeable. In order to prove Lemma~\ref{lemma:sufficent_condition}, we need another result.

\begin{lemma}\label{lemma:prob_in_hyperplane_0}
Let $X_1,\ldots,X_n$ be random vectors in $\R^d$ as above. Then, for any affine hyperplane $H\subseteq \R^d$ and any numbers $a_1,\ldots,a_n\in\R$ (not all equal to $0$), it holds that
\begin{align*}
\bP\big[a_1X_1+\ldots+a_nX_n\in H\big]=0.
\end{align*}
\end{lemma}

\begin{proof}
Suppose that $H=\{v\in\R^d:\langle v,w\rangle = z\}$ for some $w\in\R^d\backslash\{0\}$ and $z\in\R$. Now, since $(X_1,\ldots,X_n)$ has a joint $\mu^n$-density, the conditional $\mu$-density of $X_j$, conditioned on the event that
\begin{align*}
(X_1,\ldots,X_{j-1},X_{j+1},\ldots,X_n)=(x_1,\ldots,x_{j-1},x_{j+1},\ldots,x_n),
\end{align*}
exists and we denote it by $f(x_j|x^{(j)})$, for any $j\in\{1,\ldots,n\}$. For ease of notation we denote the vector on the left-hand side by $X^{(j)}$ and the vector on the right hand side by $x^{(j)}$. Now, choose $j\in\{1,\ldots,n\}$ such that $a_j\neq 0$. We obtain
\begin{align*}
\bP\big[a_1X_1+\ldots+a_nX_n\in H\big]
&=\bP\bigg[\Big\langle \sum_{i\neq j}a_iX_i,w\Big\rangle+\langle X_j,w\rangle=z\bigg]\\
&=\int_{(\R^d){n-1}}\bP\bigg[\langle X_j,w\rangle=z-\Big\langle \sum_{i\neq j}a_ix_i,w\Big\rangle\Big|X^{(j)}=x^{(j)}\bigg]\,\mu^{n-1}(\dint(x^{(j)})).
\end{align*}
Defining $y(x^{(j)}):=-\langle \sum_{i\neq j}a_ix_i,w\rangle$ and the affine hyperplane $H'(x^{(j)}):=\{v\in\R^d:\langle v,w\rangle = z-y(x^{(j)})\}$ yields
\begin{align*}
\bP\big[a_1X_1+\ldots+a_nX_n\in H\big]
=\int_{(\R^d){n-1}}\int_{H'(x^{(j)})}f\big(x_j|x^{(j)}\big)\,\mu(\dint x_j) \,\mu^{n-1}(\dint(x^{(j)}))=0,
\end{align*}
where we used that $\mu$ assigns measure $0$ to any affine hyperplane. This completes the proof.
\end{proof}

\begin{proof}[Proof of Lemma~\ref{lemma:sufficent_condition}]
In order to prove that $X_1,\ldots,X_n$, or rather their partial sums $S_1,\ldots,S_n$, satisfy the general position assumption (GP), we fix indices $1\le i_1<\ldots<i_d\le n$ and aim to prove that  $S_{i_1},\ldots,S_{i_d}$ are linearly independent with  probability $1$. If $S_{i_1},\ldots,S_{i_d}$ are linearly dependent, there exist numbers $\lambda_1,\ldots,\lambda_d\in\R$ that do not vanish simultaneously and satisfy
\begin{align*}
0=\lambda_{1}S_{i_1}+\ldots+\lambda_{d}S_{i_d}.
\end{align*}
Equivalently, we can say that there exist numbers $\lambda_1,\ldots,\lambda_d$ that do not vanish simultaneously such that
\begin{multline} \label{eq:blocks}
0=(\lambda_{1}+\ldots+\lambda_{d})(X_1+\ldots+X_{i_1})+(\lambda_{2}+\ldots+\lambda_{d})(X_{i_1+1}+\ldots+X_{i_2})\\+\ldots+\lambda_{d}(X_{i_{d-1}+1}+\ldots+X_{i_d}).
\end{multline}
Defining $k=\max\{j\in\{1,\ldots,d\}:\lambda_{j}\neq 0\}$, the last equation can be rewritten as
\begin{multline*}
X_{i_{k-1}+1}+\ldots+X_{i_k}=-\frac{\lambda_{1}+\ldots+\lambda_{k}}{\lambda_{k}}(X_1+\ldots+X_{i_1})-\frac{\lambda_{2}+\ldots+\lambda_{k}}{\lambda_{k}}(X_{i_1+1}+\ldots+X_{i_2}) \\-\ldots-\frac{\lambda_{{k-1}}+\lambda_{k}}{\lambda_{k}}(X_{i_{k-2}+1}+\ldots+X_{i_{k-1}}).
\end{multline*}
Thus, it follows that
\begin{align}\label{eq:event_probab_zero}
X_{i_{k-1}+1}+\ldots+X_{i_k}\in \lin\{X_1+\ldots+X_{i_1},\ldots,X_{i_{k-2}+1}+\ldots+X_{i_{k-1}}\},
\end{align}
where the dimension of the linear hull on the right-hand side is not greater than $k-1\le d-1$.
Conditioned on $X_1,\ldots, X_{i_{k-1}}$, the tuple $(X_{i_{k-1}+1},\ldots,X_{i_k})$ has a joint $\mu^{i_k-i_{k-1}}$-density on $(\R^d)^{i_k-i_{k-1}}$ and we can apply Lemma~\ref{lemma:prob_in_hyperplane_0}. Integrating, we see that the probability of the event~\eqref{eq:event_probab_zero} is $0$.

To prove the second claim, we follow the above reasoning with $X_i$ replaced by $\widetilde X_i$, for $i=1,\dots,n$.
To this end, suppose that there are indices $1\le i_1<\ldots<i_d\le n-1$ such that $\widetilde S_{i_1},\dots,\widetilde S_{i_d}$ are linearly dependent. By inserting the $\widetilde X_i$'s instead of the $X_i$'s into~\eqref{eq:blocks} and applying some elementary transformations, we infer that there exist $\lambda_1,\dots,\lambda_d\in\R$ that do not vanish simultaneously such that
\begin{align*}
0=(X_1+\ldots+X_{i_1})\mu_1+(X_{i_1+1}+\ldots+X_{i_2})\mu_2+\ldots+(X_{i_{d}+1}+\ldots+X_n)\mu_{d+1},
\end{align*}
where
\begin{align*}
\mu_{d+1}:=-\frac{\lambda_1i_1+\ldots+\lambda_di_d}{n},
\quad
\mu_i:=\lambda_i+\lambda_{i+1}+\ldots+\lambda_d-\frac{\lambda_1i_1+\ldots+\lambda_di_d}{n},\quad i=1,\dots,d.
\end{align*}
It is easy to check that the coefficients $\mu_1,\dots,\mu_{d+1}$ do not vanish simultaneously and satisfy the linear relation
\begin{align*}
a_1\mu_1+a_2\mu_2+\ldots+a_{d+1}\mu_{d+1}=0
\end{align*}
for the positive numbers $a_1:=i_1,a_2:=i_2-i_1,\dots,a_d:=i_d-i_{d-1},a_{d+1}:=n-i_d$. This yields
\begin{align*}
0=\frac{X_1+\ldots+X_{i_1}}{a_1}\cdot\mu_1a_1+\frac{X_{i_1+1}+\ldots+X_{i_2}}{a_2}\cdot\mu_2a_2+\ldots+\frac{X_{i_{d}+1}+\ldots+X_n}{a_{d+1}}\cdot\mu_{d+1}a_{d+1},
\end{align*}
which implies that there is a $j\in\{1,\dots,d+1\}$ such that
\begin{align}\label{eq:event}
\frac{X_{i_{j-1}+1}+\ldots+X_{i_j}}{a_j}\in\aff\left\{\frac{X_{i_{k-1}+1}+\ldots+X_{i_k}}{a_k}:k\in\{1,\dots,d+1\}\backslash\{j\} \right\},
\end{align}
with the convention that $i_0:=1$ and $i_{d+1}:=n$. Note that $\aff M$ denotes the affine hull of a set $M\subseteq\R^d$ and in our case the affine hull on the right-hand side is at most of dimension $d-1$. Following the arguments of the previous case, we can condition on all values of $X_i$ for $i\in\{1,\dots,n\}\backslash\{i_{j-1}+1,\dots,i_j\}$ and then apply Lemma~\ref{lemma:prob_in_hyperplane_0} to conclude that the event in~\eqref{eq:event} has probability $0$. This completes the proof.
\end{proof}

\section*{Acknowledgement}
Supported by the German Research Foundation under Germany's Excellence Strategy  EXC 2044 -- 390685587, \textit{Mathematics M\"unster: Dynamics - Geometry - Structure}  and by the DFG priority program SPP 2265 \textit{Random Geometric Systems}.

\vspace*{0.5cm}
\bibliography{bib}
\bibliographystyle{abbrv}



\end{document}